\documentclass[12pt]{amsart}

\usepackage{amssymb}
\usepackage{mathrsfs}
\usepackage{xypic}
\usepackage{graphicx}
\usepackage{psfrag}
\usepackage{cite}
  
\theoremstyle{plain}
\newtheorem{theorem}{Theorem}[section]
\newtheorem{lemma}[theorem]{Lemma}
\newtheorem{proposition}[theorem]{Proposition}
\newtheorem{corollary}[theorem]{Corollary}
\newtheorem{Theorem}{Theorem}

\theoremstyle{definition}
\newtheorem{definition}{Definition}[section]

\theoremstyle{remark}
\newtheorem{remark}{Remark}[section]
\newtheorem{example}[remark]{Example}

\renewenvironment{proof}[1]{\vspace*{.1in}\noindent{\bf
Proof{#1}. \/}}{\qed\vspace{3ex}} 
		     
\newcommand{\Lk}{\operatorname{Lk}}

\newcommand{\Id}{\operatorname{Id}}

\newcommand{\St}{\operatorname{St}}
\newcommand{\Nb}{\operatorname{Nb}}

\renewcommand{\Vert}{\operatorname{Vert}}

\newcommand{\ord}{\operatorname{Ord}}
\newcommand{\card}{\operatorname{Card}}
\newcommand{\supp}{\operatorname{Supp}}
\newcommand{\desc}{\operatorname{Desc}}

\newcommand{\<}{\langle}
\renewcommand{\>}{\rangle}

\newcommand{\bbR}{\mathbb R}

\newcommand{\bbZ}{\mathbb Z}
\newcommand{\bbQ}{\mathbb Q}

\newcommand{\bsigma}{\boldsymbol{\sigma}}
\newcommand{\btau}{\boldsymbol{\tau}}
\newcommand{\brho}{\boldsymbol{\rho}}
\newcommand{\bchi}{\boldsymbol{\chi}}

\newcommand{\bt}{\mathbf t}

\renewcommand{\cL}{\mathcal L}
\newcommand{\cA}{\mathcal A}
\newcommand{\cS}{\mathcal S}
\renewcommand{\cR}{\mathcal R}
\renewcommand{\cH}{\mathcal H}

\newcommand{\cP}{\mathcal P}

\renewcommand{\hat}[1]{\widehat{#1}}
\renewcommand{\bar}[1]{\overline{#1}}
\renewcommand{\tilde}[1]{\widetilde{#1}}

\catcode`\@=11
\def\@secnumfont{\bfseries}
\def\section{\@startsection{section}{1}%
  \z@{.7\linespacing\@plus\linespacing}{.5\linespacing}%
  {\normalfont\centering\bfseries}}
\def\subsection{\@startsection{subsection}{2}%
  \z@{.5\linespacing\@plus.7\linespacing}{-.5em}%
  {\normalfont\bfseries}}
\catcode`\@12

\title[Rationality and Reciprocity for Coxeter groups]{Rationality and reciprocity for the greedy\\ normal form of a
  Coxeter group}

\author{Richard Scott}
\address{Department of Mathematics and Computer Science\\
Santa Clara University\\
Santa Clara, CA  95053}
\email{rscott@math.scu.edu}

\thanks{The author thanks MSRI for its support and hospitality during
  the writing of this paper.}  

\begin{document}

\begin{abstract}
We show that the characteristic series for the greedy normal form of a
Coxeter group is always a rational series, and prove a reciprocity
formula for this series when the group is right-angled and the nerve
is Eulerian.  As corollaries we obtain many of the known  rationality
and reciprocity results for the growth series of Coxeter groups as
well as some new ones.  
\end{abstract}

\maketitle
\section{Introduction}

{\bf Background.}
Let $G$ be a finitely generated group with a finite generating set
$S$.  The {\em   growth series} of $G$ relative to $S$ is the power
  series  
\[\gamma(t)=\sum_{g\in G}t^{|g|}\]
where $|g|$ denotes the word length of $g$ induced by $S$.  Growth series are an important measure of
complexity and size for infinite groups.  If $G$ satisfies certain
algorithmic properties (for example, has a normal form recognized by a
finite state automaton), then $\gamma$ is known to be a rational
function. Moreover, if $G$ is a discrete group acting
geometrically on a manifold or cell complex $M$ and $S$ is suitably
related to the geometry of $M$, then one often finds interesting
connections with the topology of $M$.  Two notable examples are cases
where (1) special values of $\gamma(t)$ are related to the Euler-Poincar\'{e}
characteristic of $M$ or $G$  and (2) reciprocity formulas for
$\gamma$ (see, e.g., \cite{Serre, CD, Dymara} for Coxeter groups,
\cite{FP} for Fuchsian groups).  For example, in the
case where $W$ is a Coxeter group and $S$ is the standard generating
set, Serre \cite{Serre} showed that $\bchi(W)=1/\gamma(1)$ and
that for affine Coxeter groups $\gamma(1/t)=\pm\gamma(t)$.  

To describe more general formulas involving Coxeter groups, we recall
some standard terminology and constructions. Let $W$
be a Coxeter groups with standard generating set $S$. Given
a subset $\sigma\subset S$, let $W_{\sigma}$ denote the {\em parabolic
  subgroup} of $W$ generated by $\sigma$ (by convention
$W_{\emptyset}=\{1\}$).  Let $N=N(W,S)$ be the set of all {\em
  spherical subsets}
$\sigma\subset S$; i.e., $N$ consists of all $\sigma$ such that
$W_{\sigma}$ is finite. Each such finite $W_{\sigma}$ is itself a Coxeter
group and has a unique element of longest length, 
which we denote by $w_{\sigma}$.  The collection $N$, partially
ordered by inclusion, is an abstract simplicial complex, called the
{\em nerve} of $W$. Let $W\cS$ denote the collection of cosets
$\{wW_{\sigma}\;|\; w\in W,\;\sigma\in N\}$ with partial order
defined by inclusion.  The {\em Davis complex}
\cite{Davis-annals} associated to $(W,S)$ is the geometric realization
of $W\cS$ and is usually denoted by $\Sigma$. It is a contractible
complex on which $W$ acts properly with finite stabilizers.   

In \cite{CD}, Charney and Davis generalized the reciprocity for affine
Coxeter groups, showing that that $\gamma(1/t)=\pm\gamma(t)$
whenever the the nerve $N$ is an {\em Eulerian sphere}, i.e., has the property
that the link of every simplex $\sigma$ has the Euler characteristic
of a sphere (of the appropriate dimension).  More recently, Dymara  
\cite{Dymara} has shown that for any positive real number $q$, one has 
$\bchi^q_2(\Sigma)=1/\gamma(q)$ where $\bchi^q_2(\Sigma)$ 
denotes the ``$q$-weighted'' $\ell_2$-Euler characteristic of $\Sigma$.
When $\Sigma$ is a (generalized homology) manifold, he  also obtains the
reciprocity formula $\gamma(1/q)=\pm\gamma(q)$ as a 
consequence of Poincar\'{e} duality for weighted $\ell_2$-cohomology.

It was observed by Serre that for Coxeter groups, there is a natural
extension of the growth series to a multivariable series $\gamma(\bt)$
where $\bt$ is a tuple of indeterminants, one for each conjugacy class
of generators in $S$.  The $\ell_2$-Euler characteristic and the
reciprocity formulas above also have natural interpretations in terms
of this multivariable series $\gamma(\bt)$ (see \cite{DDJO, Davis-book}).

{\bf Rationality Results.}
Although explicit rational formulas for the growth series of a Coxeter
group were known to Steinberg \cite{Steinberg}, rationality itself
follows on general principles from the fact that Coxeter groups are
automatic (the relevant property of an automatic group is that it has
a normal form that is recognized by a finite state automaton). The
proof of automaticity for Coxeter groups was completed by Brink and
Howlett in \cite{BH}, but in the special case of right-angled Coxeter 
groups, automaticity (in fact, bi-automaticity) also follows from the
more recent results of Niblo and Reeves \cite{NR} and the fact that
the Davis complex is a CAT($0$) cube complex.  The normal form  for
groups acting on cube complexes turns out to be different from the
normal form recognized by the Brink/Howlett automaton, and is in some
sense more canonical (see Section~\ref{s:cox-rec}, below).  This canonical
normal form for right-angled Coxeter groups is a special case of a
normal form that makes sense for any Coxeter group.  For each element
$w\in W$, this normal form specifies a unique product representation
of the form  $w=w_{\sigma_n}\cdots w_{\sigma_1}$ where 
each letter $w_{\sigma_i}$ is an element of longest length in
the (finite) parabolic subgroup $W_{\sigma_i}$.  The left-most
letter $w_{\sigma}$ appearing in the representation for $w$ has the
property that $\sigma$ is the largest subset such that
\[|w|=|w_{\sigma}|+|w_{\sigma}^{-1}w|.\]
For this reason, we call this normal form the {\em (left) greedy
  normal form} and denote it by $\cL(W,S)$. 

The main purpose of this article is to describe rationality and
reciprocity theorems for $\cL(W,S)$.  To avoid confusion between words in
$\cL(W,S)$ and group elements in $W$, we let $\cA^*$ denote the free monoid
on the set $\cA=\{\sigma\in N \;|\; \sigma\neq\emptyset\}$ and regard
$\cL(W,S)$ as a subset of $\cA^*$ (in other words, $\cL(W,S)$ is a {\em language}
with respect to the {\em alphabet} $\cA$).  Given any language
$\cL\subset\cA^*$, we let $\chi_{\cL}$ denote the {\em characteristic
  series} defined by 
\[\chi_{\cL}=\sum_{\alpha\in\cL}\alpha.\]
We regard it as a formal power series over $\bbQ$ (with the elements of
$\cA$ representing noncommuting formal variables). 
Such a series is {\em rational} if it can be obtained from a finite
set of polynomials using a finite sequence of additions,
multiplications and quasi-inversions (see Section~\ref{s:rational}).
It is a classical result in automata theory that $\cL$ is a regular language
(i.e., recognized by a finite state automaton) if and only if $\chi_{\cL}$
is rational (see, e.g., \cite{Schutzenberger}).  Thus, it follows from
Niblo and Reeves' proof of automaticity for groups acting on CAT($0$) 
cube complexes, that in the case of a right-angled Coxeter group the
characteristic series for the greedy normal form $\cL(W,S)$ is
rational.  More generally, we prove the following.    

\begin{Theorem} \label{thm:A}
Let $W$ be any Coxeter group and let $\cL=\cL(W,S)$ be its greedy normal
form.  Then the characteristic series $\chi_{\cL}$ is rational.
\end{Theorem}

Since the greedy normal form consists of minimal length words,
substituting the appropriate commuting parameters for each letter
$\sigma\in\cA$ converts $\chi_{\cL}$ into the usual growth series
(single or multi-parameter) of the Coxeter group $W$ relative to the
generating set $S$.  Thus, rationality of 
the noncommutative series $\chi_{\cL}$ implies rationality for the
usual growth series $\gamma(t)$ and its multivariate versions.

The proof of Theorem~\ref{thm:A} (in Section~\ref{s:cox-rec}) implicitly constructs
an automaton that recognizes the greedy normal form.   The states of
the automaton are in bijection with certain regions (of the Tits'
cone) cut out by minimal hyperplanes, and in this respect the
automaton is similar to the canonical automaton described in
Bj\"{o}rner and Brenti \cite[page 120]{BB}. The latter, however, has the set of all reduced
expressions (as words in the standard generators) as its accepted
language, and thus does not have the uniqueness property -- a given
group element has many reduced expressions.   On the other hand, the
greedy normal form (defined over the larger alphabet $\cA$) does have
the uniqueness property, and our automaton appears to be new.    

{\bf Reciprocity Results.} 
To make sense of reciprocity in the context of multivariate
noncommutative power series, we work over the ring of formal
Laurent series.  In particular, in this ring each formal parameter
$\sigma\in\cA$ has an inverse $\sigma^{-1}$.  We use the 
fact that for any regular language, the characteristic series
$\chi_{\cL}$ has a representation of the form 
\[\chi_{\cL}=A(I-Q)^{-1}B\]
where $A$ is a row vector with entries in $\bbQ$, $B$ is a column
vector with entries in $\bbQ$, and $Q$ is a square matrix whose entries
are linear combinations of the variables $\sigma\in\cA$. (Roughly
speaking, with respect to the automaton 
that recognizes $\cL$, $A$ records the accept states, $B$ records the
start state, and $Q$ encodes the transition function.)  By replacing each
$\sigma$ appearing in $Q$  
with $\sigma^{-1}$, we obtain a new matrix $\bar{Q}$, and we define
the reciprocal $\chi_{\cL}^*$ to be 
\[\chi_{\cL}^*=A(I-\bar{Q})^{-1}B\]
{\em provided that the matrix $I-\bar{Q}$ is invertible} over the ring
of Laurent series.  It follows from Schutzenberger \cite{Schutzenberger} that
such a series, if it exists, does not depend on the choice of $A$,$B$,
and $Q$ (see Proposition~\ref{prop:recip}, below).  The main theorem of this paper is
the following.

\begin{Theorem}\label{thm:B}
Let $W$ be a right-angled Coxeter group and let $\cL=\cL(W,S)$ be its
greedy normal form.  If the nerve $N(W,S)$ is an Eulerian sphere of
dimension $k$, then $\chi_{\cL}^*$ exists and 
\[\chi^*_{\cL}=(-1)^{k+1}\chi_{\cL}.\]  
\end{Theorem}

Again, appropriate substitutions for the letters $\sigma\in\cA$ yield
all of the known reciprocity formulas {\em in the right-angled
  cased}.  In fact, in the generality we work, we also obtain a new
reciprocity formula for the {\em complete growth series} of a
right-angled Coxeter group.  By definition, the complete growth series
of $G$ with respect to a generating set $S$ is the series  
\[\tilde{\gamma}(t)=\sum_{g\in G}gt^{|g|},\]
a power series in $t$ with coefficients in the (noncommutative) group
ring $\bbQ[G]$.  The complete growth series has been studied by
several authors \cite{Liardet,GN,GdlH}.  Rationality for hyperbolic groups and
formulas for surface groups were obtained by Grigorchuk and Nagnibeda in
\cite{GN}.  For Coxeter groups, rationality of the complete growth
series is known \cite{Changey,Mamaghani}, but also follows from our
Theorem~\ref{thm:A} and the substitutions $\sigma\rightarrow
w_{\sigma}t^{|\sigma|}$ (for all $\sigma\in\cA$) into $\chi_{\cL}$.
More notable, however, is the fact that this same substitution yields
the reciprocity formula
\[\tilde{\gamma}(1/t)=(-1)^{k+1}\tilde{\gamma}(t)\]
for the complete growth series of a right-angled Coxeter group whose
nerve is an Eulerian $k$-sphere. 

Theorem~\ref{thm:B} fails to hold in general for non right-angled
Coxeter groups with Eulerian nerve (even though the usual growth
series does satisfy reciprocity \cite{CD}).  In these cases, it would
be interesting to find an identity involving $\chi_{\cL}$ that does
specialize to the usual reciprocity formulas under an appropriate
substitution. 

The proof of Theorem~\ref{thm:B} reduces to showing that for a
simplicial complex $K$ a certain matrix $J$ is an involution when
$K$ is Eulerian.  If $\cP(K)$ denotes the poset of faces of $K$, then
$J$ is a $\cP(K)\times\cP(K)$ matrix that takes nonzero values $\pm 1$
if and only if the corresponding pair of simplices are sufficiently
far apart in $K$.  For this reason we call $J$ the {\em anti-incidence
  matrix} for $K$.  The proof that $J$ is an involution is fairly
technical, but may be of independent interest to combinatorialists.
For this reason, we include it as a separate section. 

{\bf Organization.}
Section~\ref{s:series} introduces notation and standard facts
for noncommutative series.  Our reciprocity formulas require working
with formal inverses of generators, so we work over the ring of formal
noncommutative Laurent series.   Sections~\ref{s:rational} and
\ref{s:reciprocal} discusses rationality and reciprocity for formal
series.  In Section~\ref{s:coxeter} we describe the greedy normal form
for a Coxeter group.  In Section~\ref{s:cox-rat} we show that the
characteristic series for the greedy normal form is rational,
essentially by describing a finite state automaton that recognizes the
language $\cL(W,S)$.  Section~\ref{s:Eulerian} discusses some
properties of Eulerian simplicial complexes, the most important of
which being that the anti-incidence matrix is an involution.  In
Section~\ref{s:cox-rec}, we combine the rationality results and
Eulerian complex results to prove our reciprocity theorem. Finally, in
Section~\ref{s:growthseries}, we apply our main theorems to obtain
rationality and reciprocity results for growth series of Coxeter
groups relative to both the standard generating set $S$ and the larger
generating set $A=\{w_{\sigma}\;|\;\sigma\in\cA\}$.

\section{Noncommutative formal series}\label{s:series}

In this section, we recall some basic facts from the theory of
non-commutative series.    

Let $\cA$ be a finite set, called the {\em alphabet}.  We let $\cA^+$
and $\cA^-$ denote two copies of $\cA$, and let $\cA^{\pm}$ denote the
disjoint union $\cA^+\cup \cA^-$.  For each element $ a\in\cA$,
we denote the corresponding elements in $\cA^+$ and $\cA^-$ by $a$
and $a^{-1}$, respectively.  We let $\cA^*$ (respectively,
$(\cA^{\pm})^*$) denote the free monoid generated by $\cA$ (resp.,
$\cA^{\pm}$) consisting of {\em words} over $\cA$ (resp., over
$\cA^{\pm}$).  A word over $\cA^{\pm}$ is {\em reduced} if it is of the form 
\[\alpha=a_1^{m_1}\cdots a_n^{m_n}\]
where $a_i\in\cA$, $m_i\in\{\pm 1\}$, and there are no consecutive
pairs of the form $aa^{-1}$ or $a^{-1}a$
appearing in the expression.  We let $F(\cA)$ denote the set of
reduced words over $(\cA)^{\pm}$, and endow it with the usual
multiplication to form the free group over $\cA$.  We then define the
{\em order} of an element $\alpha=a_1^{m_1}\cdots a_n^{m_n}\in
F(\cA)$, which we denote by $\ord(\alpha)$, to be the total number of
$-1$'s appearing as exponents, that is if $\alpha=a_1^{m_1}\cdots
a_n^{m_n}$, then $\ord(\alpha)=\card\{i\:|\;
m_i=-1\}$.  In particular, $\cA^*$ can be identified with the
submonoid of $F(\cA)$ consisting of all elements $\alpha$ such that
$\ord(\alpha)\geq 0$.  

We work over the rational numbers $\bbQ$, and consider all
formal series $\lambda$ of the form 
\[\lambda=\sum_{\alpha\in F(\cA)}r_{\alpha} \alpha.\]
where $r_{\alpha}\in\bbQ$.  The element $r_{\alpha}$ is called the {\em coefficient}
of $\alpha$ in $\lambda$ and will also be denoted by
$\<\lambda,\alpha\>$.  We define the {\em support} of a series $\lambda$ to be the set  
\[\supp(\lambda)=\{\alpha\in F(\cA)\;|\; \<\lambda,\alpha\>\neq 0\},\]
and we define the {\em order} of $\lambda$ to be 
\[\ord(\lambda)=\inf\{\ord(\alpha)\;|\;\alpha\in\supp(\lambda)\}.\]

The sum of two series is defined in the usual way
\[\lambda_1+\lambda_2=\sum_{\alpha\in
  F(\cA)}(\<\lambda_1,\alpha\>+\<\lambda_2,\alpha\>)\alpha\]
but the formal product 
\begin{equation}\label{eq:product} 
\lambda_1\lambda_2=\sum_{\alpha\in
  F(\cA)}\left(\sum_{\alpha_1\alpha_2=\alpha}\<\lambda_1,\alpha_1\>\<\lambda_2,\alpha_2\>\right)\alpha
\end{equation}
is not always well-defined (the inner sum might be infinite).  
To remedy this, we define a {\em Laurent series over $\cA$} to be any
formal series $\lambda$ such that 
$\ord(\lambda)>-\infty$.  If $\lambda_1$ and $\lambda_2$ are both
Laurent series, then the inner sum in (\ref{eq:product}) will always
be a finite sum and the resulting product $\lambda_1\lambda_2$ will
also have finite order.  Hence the set of formal Laurent series forms
an (associative) ring with unit. We denote this ring by $\bbQ((\cA))$.  

A Laurent series $\lambda$ is called a {\em power series} if
$\supp(\lambda)\subset \cA^*$, a {\em Laurent polynomial} if
$\supp(\lambda)$ is finite, and a {\em   polynomial} if it is both a
Laurent polynomial {\em and} a power series.  The {\em degree} of a
polynomial $\lambda$ is the length of the longest word in 
$\supp(\lambda)$.  A polynomial of degree $1$ is called {\em linear},
and a polynomial with all terms of the same degree is called {\em
homogeneous}.  We let $\bbQ\<\<\cA\>\>$ (respectively, $\bbQ(\cA)$,
$\bbQ\<\cA\>$) denote the set of all power series (resp., Laurent
polynomials, polynomials).  All three of these sets are subrings of 
$\bbQ((\cA))$.

\section{Rational power series}\label{s:rational}

A power series $\lambda\in\bbQ\<\<\cA\>\>$ is {\em quasi-regular} if the
constant term is zero, i.e., if $1\not\in\supp(\lambda)$.  For a
quasi-regular element $\lambda$, we define its {\em quasi-inverse}
$\lambda^+$ by  
\[\lambda^+=\sum_{n=1}^{\infty}\lambda^n=\lim_{N\rightarrow\infty}\sum_{n=1}^N\lambda^N.\]
Note that if $\lambda$ is quasi-regular, then $1-\lambda$ is
invertible and $(1-\lambda)^{-1}=1+\lambda^+$.  A power series
$\lambda\in\bbQ\<\<\cA\>\>$ is {\em rational} if it can be obtained from
a finite set of polynomials by a finite sequence of additions,
multiplications, and quasi-inversions.    

The notation and terminology above also extends to vectors and
matrices.  Given a ring $R$, we let $R^{m\times n}$ denote the set of
$m\times n$ matrices with entries in $R$.  As in the case for
quasi-regular power series, if a matrix $Q\in\bbQ\<\<\cA\>\>^{n\times n}$ has all
quasi-regular entries, then $I-Q$ is invertible (over $\bbQ\<\<\cA\>\>$)
with inverse given by 
\[(I-Q)^{-1}=I+Q^+=I+Q+Q^2+Q^3+\cdots.\]

Important examples of power series come from the theory of formal
languages.  By definition a {\em language} over $\cA$ is any subset
$\cL\subset \cA^*$.  Given a language over $\cA$, we define its {\em
  characteristic} series $\chi_{\cL}\in\bbQ\<\<\cA\>\>$ by 
\[\chi_{\cL}=\sum_{\alpha\in\cL}\alpha.\]
A language $\cL$ is {\em regular} if it is the language accepted by a
finite state automaton.  This turns out to be equivalent to its
characteristic series $\chi_{\cL}$ being rational.  In fact, more
generally, we have the following standard characterization of rational power
series, due to Schutzenberger \cite{Schutzenberger,SS}.

\begin{theorem}\label{thm:psrep}
A power series $\lambda\in\bbQ\<\<\cA\>\>$ is rational if and only if
there exist a matrix $Q\in\bbQ\<\cA\>^{n\times n}$ with linear
homogeneous entries and vectors $A\in\bbQ^{1\times n}$ and
$B\in\bbQ^{n\times 1}$ such that 
\[\lambda=A(I-Q)^{-1}B.\]
\end{theorem}

\begin{proof}{}
Any series of such a form $A(I+Q^+)B$ is rational by Theorem~1.2 in
\cite{SS}.  (The proof there shows this if $A=[1\;0\cdots\;0]$.  The
general case follows by a change of basis.)  Conversely, any rational
series is of this form by Theorem~1.4 in \cite{SS}.  
\end{proof}

\begin{remark}
Series of the form $A(I-Q)^{-1}B$ are usually called {\em
recognizable}, thus Schutzenberger's Theorem says that a series is rational
if and only if it is recognizable.  The connection to regular
languages is fairly straightforward.  Given a finite state automaton,
let $\cS$ denote the set of states and let
$\mu:\cA\times\cS\rightarrow\cS$ denote the transition function.
Let $A$ be the $1\times\cS$ row
vector whose entry corresponding to $s\in S$ is $1$ if $s$ is an
accept state and is $0$ otherwise.  Let $B$ be the  $\cS\times 1$ column
vector whose entry corresponding to $s\in\cS$ is $1$ if $s$ is the
start state and is $0$ otherwise.  And let $Q$ be the $\cS\times\cS$
matrix whose $s,t$-entry is $\sum_{\mu(a,t)=s}a$.  Then the
characteristic series for the language $\cL$ recognized by the
automaton is   
\[\chi_{\cL}=A(I-Q)^{-1}B=A(I+Q+Q^2+Q^3+\cdots)B.\]
See, e.g., \cite{Reutenauer} for more details.
\end{remark} 

Given a rational power series $\lambda$, a triple $(A,Q,B)$ as in
Theorem~\ref{thm:psrep} is called a {\em representation} for $\lambda$.
The {\em dimension} of a representation is the number $n$.  Two
representations $(A,Q,B)$ and $(A',Q',B')$ are {\em
  equivalent} if there exists an invertible matrix $P\in\bbQ^{n\times
  n}$ such that $A'=AP^{-1}$, $Q'=PQP^{-1}$, and $B'=PB$.  A
representation for $\lambda$ is {\em minimal} if it has them minimum
dimension among all representations for $\lambda$.   

\begin{example}\label{ex:square-free}
Let $\cA=\{x,y\}$ and let $\cL\subset\cA^*$ be the language consisting of
all square free words in $x$ and $y$.  Then the characteristic series
for $\cL$ is  
\[\chi_{\cL}=1+x+y+xy+yx+xyx+yxy+\cdots.\]
This series is rational since it can be written as 
\[\chi_{\cL}=A(I-Q)^{-1}B=A(I+Q+Q^2+\cdots)B\]
where 
\[A=\left[\begin{array}{ccc}1&1&1\end{array}\right],\hspace{.5in}
Q=\left[\begin{array}{ccc}0&0&0\\x&0&x\\y&y&0\end{array}\right],\hspace{.5in}
B=\left[\begin{array}{c}1\\0\\0\end{array}\right].\]
In other words, $(A,Q,B)$ is a representation for $\chi_{\cL}$.  Using
row reduction to invert $I-Q$, one obtains
\[(I-Q)^{-1}=\left[\begin{array}{ccc}
1&0&0\\
x+x(1-yx)^{-1}y(1+x)&1+x(1-yx)^{-1}y&x\\
(1-yx)^{-1}y(1+x)&(1-yx)^{-1}y&(1-yx)^{-1}\end{array}\right].\]
This gives a rational expression for the power series $\chi_{\cL}$:
\[\chi_{\cL}=A(I-Q)^{-1}B=1+x+x(1-yx)^{-1}y(1+x)+(1-yx)^{-1}y(1+x).\] 
\end{example}

\section{The reciprocal of a rational power series}\label{s:reciprocal}

The goal of section is to make sense (when possible) of the series
obtained from a rational power series by replacing all letters with
their inverses.  
 
Let $\lambda\mapsto\bar{\lambda}$ denote the involution on the set of
Laurent polynomials defined by replacing each letter $x\in \cA$
with $x^{-1}$ and each letter $x^{-1}$ with $x$.  Similarly, for any
matrix $M$ with Laurent polynomial entries, we define $\bar{M}$ to be
the same matrix with each entry $\lambda$ replaced by $\bar{\lambda}$.

Now suppose $\lambda$ is a {\em rational} power series, and that
$(A,Q,B)$ is a representation for $\lambda$.  Suppose further
that the matrix $I-\bar{Q}$ is invertible over the ring $\bbQ((\cA))$.
Then we obtain a new Laurent series   
\[\lambda^*(A,Q,B)=A(I-\bar{Q})^{-1} B.\]

\begin{proposition}\label{prop:recip}  
The series $\lambda^*(A,Q,B)$, if it exists, does not depend on the choice of
representation $(A,Q,B)$.   
\end{proposition}

For the proof of Proposition~\ref{prop:recip}, we use
the following result due to Schutzenberger \cite[Theorem
  III.B.1]{Schutzenberger}.     

\begin{lemma}\label{lem:standard form}
Let $\lambda$ be a rational series in $\bbQ\<\<\cA\>\>$, and let 
$(A_0,Q_0,B_0)$ be a minimal representation for $\lambda$.  Then any
other representation for $\lambda$ is equivalent to one of the form
$(\hat{A},\hat{Q},\hat{B})$ where these matrices have the block form
\[\hat{A}=\left[\begin{array}{ccc}* & A_0 & 0\end{array}\right]\;\hspace{.5in} 
\hat{Q}=\left[\begin{array}{ccc}Q_1&0&0\\ * &Q_0& 0\\ * & * & Q_2\end{array}\right]\;\hspace{.5in} 
\hat{B}=\left[\begin{array}{ccc}0\\B_0\\ *\end{array}\right].\]
In particular, any two minimal representations are equivalent.
\end{lemma}

\begin{proof}{ of Proposition~\ref{prop:recip}}
Let $(A,Q,B)$ be a representation for $\lambda$ and let
$(A_0,Q_0,B_0)$ and $(\hat{A},\hat{Q},\hat{B})$ be as in
Lemma~\ref{lem:standard form}.  It follows easily from the definitions
of equivalent representations that
$\lambda^*(A,Q,B)=\lambda^*(\hat{A},\hat{Q},\hat{B})$.  Calculating
the latter, we have
\begin{eqnarray*}
\lambda^*(\hat{A},\hat{Q},\hat{B}) & = & 
\left[\begin{array}{ccc}* & A_0 & 0\end{array}\right]
\left(I-\overline{\left[\begin{array}{ccc}Q_1&0&0\\ * &Q_0& 0\\ * & * &
      Q_2\end{array}\right]}\right)^{-1}
\left[\begin{array}{ccc}0\\B_0\\ *\end{array}\right]\\
& = & 
\left[\begin{array}{ccc}* & A_0 & 0\end{array}\right]
\left[\begin{array}{ccc}I-\bar{Q}_1&0&0\\ \bar{*} &I-\bar{Q}_0& 0\\ \bar{*} & \bar{*} &
      I-\bar{Q}_2\end{array}\right]^{-1}
\left[\begin{array}{ccc}0\\B_0\\ *\end{array}\right]\\
& = & 
\left[\begin{array}{ccc}* & A_0 & 0\end{array}\right]
\left[\begin{array}{ccc}(I-\bar{Q}_1)^{-1}&0&0\\ ** &(I-\bar{Q}_0)^{-1}& 0\\ ** & ** &
      (I-\bar{Q}_2)^{-1}\end{array}\right]
\left[\begin{array}{ccc}0\\B_0\\ *\end{array}\right]\\
& = & 
A_0(I-\bar{Q}_0)^{-1}B_0\\
& = & 
\lambda^*(A_0,Q_0,B_0)
\end{eqnarray*}
Since any two minimal representations are equivalent, the result follows.
\end{proof}

Since the series $\lambda^*(A,Q,B)$ does not depend on the
representation, we denote it simply by $\lambda^*$ and call it the
{\em reciprocal} of $\lambda$.  

\begin{example}\label{ex:square-free-2}
Let $\chi$ be the characteristic series for the square-free words in
$x$ and $y$ as in Example~\ref{ex:square-free}.  Then $Q$ is the
matrix 
\[Q=\left[\begin{array}{ccc}0&0&0\\x&0&x\\y&y&0\end{array}\right],\]
so we have 
\begin{eqnarray*}
I-\bar{Q} & = & \left[\begin{array}{ccc}1&0&0\\-x^{-1}&1&-x^{-1}\\-y^{-1}&-y^{-1}&1\end{array}\right]\\
 & = &\left[\begin{array}{ccc}-1&0&0\\0&-x^{-1}&0\\0&0&-y^{-1}\end{array}\right]\left[\begin{array}{ccc}-1&0&0\\1&-x&1\\1&1&-y\end{array}\right]\\
& = &
 \left[\begin{array}{ccc}-1&0&0\\0&-x^{-1}&0\\0&0&-y^{-1}\end{array}\right]
\left(I-Q\right)
\left[\begin{array}{ccc}-1&0&0\\1&0&1\\1&1&0\end{array}\right]
\end{eqnarray*}
Inverting this matrix, we obtain
\[(I-\bar{Q})^{-1}= \left[\begin{array}{ccc}-1&0&0\\1&0&1\\1&1&0\end{array}\right]
\left(I-Q\right)^{-1}
\left[\begin{array}{ccc}-1&0&0\\0&-x&0\\0&0&-y\end{array}\right].\]
We then have 
\begin{eqnarray*}
\chi^* & = & \left[\begin{array}{ccc}1&1&1\end{array}\right](I-\bar{Q})^{-1}\left[\begin{array}{c}1\\0\\0\end{array}\right]\\
& = &
 \left[\begin{array}{ccc}1&1&1\end{array}\right]\left(I-Q\right)^{-1}\left[\begin{array}{c}-1\\0\\0\end{array}\right]\\
& = & -1-x-y-xy-yx-xyx-yxy-\cdots.
\end{eqnarray*}
\end{example}

\section{Coxeter groups and the greedy normal form}\label{s:coxeter}

Let $(W,S)$ be a Coxeter system as defined in \cite{Bourbaki}.  In
particular, $W$ is a group, $S$ is a generating set, and $W$ admits a
presentation of the form 
\[W=\langle S\;|\; \mbox{$(ss')^{m(s,s')}=1$ for all $s,s'\in S$}\rangle\]
where $m:S\times S\rightarrow\{1,2,\ldots,\infty\}$ is a symmetric
function such that $m(s,s')=1$ if and only if $s=s'$.  Any element
$w\in W$ has a {\em length}, denoted $|w|$, defined to be the minimal
number $n$ such that $w=s_1\cdots s_n$ with all $s_i\in S$. 

A subset $\sigma\subset S$ is called {\em spherical} if it is
either empty or generates a finite subgroup of $W$.  The {\em nerve}
of a Coxeter system $(W,S)$ is the set $ N(W,S)$
consisting of all spherical subsets of $S$.  We can regard $N=N(W,S)$ as
an abstract simplicial complex on the vertex set $S$ (so $\emptyset$
corresponds to the empty simplex).  For any nontrivial
$\sigma\in N(W,S)$, we let $W_{\sigma}$ denote the subgroup generated
by $\sigma$, and we let $w_{\sigma}$ denote the element of longest
length in $W_{\sigma}$.  Note that $w_{\sigma}$ is always an
involution (\cite[Ex.~22, p.~43]{Bourbaki}).  For any $w\in W$, we
define the {\em (left) descent set} of $w$ to be the set $\desc(w)=\{s\in S\;|\;
|sw|<|w|\}$.   A proof of the following can be found, for example in
\cite{Davis-book} (Lemma~4.7.2). 

\begin{proposition}\label{prop:longestelement}
Let $w$ be any element in $W$, and let $\sigma=\desc(w)$ be the left
descent set.  Then 
\begin{enumerate}
\item $\sigma$ is a spherical subset of $S$, and 
\item $w$ factors as $w=w_{\sigma}v$ where $|w|=|w_{\sigma}|+|v|$.
\end{enumerate}
\end{proposition}

Now let $\cA$ be the set of of all nontrivial spherical
subsets of $S$, which we call the {\em proper nerve} of $W$.  Given a word
$\alpha=\sigma_n\cdots\sigma_1\in{ \cA}^*$, let $\pi(\alpha)\in W$ be the 
product  $w_{\sigma_n}\cdots 
w_{\sigma_1}$  (and let $\pi(1)=1$).  Then $\pi:{ \cA}^*\rightarrow W$ is a
monoid homomorphism, and since all of the generators for $W$ are of
the form $w_{\sigma}$ with $\sigma=\{s\}$ a spherical subset, $\pi$ is
surjective.  

\begin{definition}
Let $(W,S)$ be a Coxeter system and let ${ \cA}$ be the proper nerve.  Then
the {\em left greedy normal form} for $W$ is the language $\cL\subset
{ \cA}^*$ consisting of all words 
$\alpha=\sigma_n\sigma_{n-1}\cdots\sigma_1\in { \cA}^*$  such that
$\sigma_i$ is the left descent set of
$w_{\sigma_i}w_{\sigma_{i-1}}\cdots w_{\sigma_1}$ (for $i=1,\ldots
n$).
\end{definition}

We then have the following.

\begin{proposition}
Let $\cL$ be the left greedy normal form for the Coxeter group
$(W,S)$.  Then $\pi$ restricts to a bijection $\cL\rightarrow W$.
Moreover, for any $\alpha=\sigma_n\cdots\sigma_1\in \cL$, we have
$|\pi(\alpha)|=|w_{\sigma_n}|+\cdots+|w_{\sigma_1}|$.  
\end{proposition}

\begin{proof}{}
The equality $|\pi(\alpha)|=|w_{\sigma_n}|+\cdots+|w_{\sigma_1}|$
follows by induction on $n$ and
Proposition~\ref{prop:longestelement}.  

To prove that $\pi$ is
injective, note first that only the trivial word maps to the
identity.  Assume by induction that any two words (in $\cL$) that map
to a group element of length $<k$ are equal, and suppose that 
$\alpha=\sigma_n\cdots\sigma_1$ and
$\alpha'=\sigma_m'\cdots\sigma'_1$ in $\cL$ both map to the same group
element $w$ with $|w|=k$.  Then $\sigma_n$ and $\sigma_m'$ are both the left
descent set of $w$, hence are equal, so $\sigma_{n-1}\cdots\sigma_1$
and $\sigma_{m-1}'\cdots\sigma'_1$ both map to the same group element
(of length $<k$).  By induction, these must be the same word.  

For surjectivity (and the length property), assume by induction that
every $w\in W$ of length $<k$ has a greedy representation
$\alpha=\sigma_n\cdots\sigma_1\in \cL$ and that
$|\pi(\alpha)|=|w_{\sigma_n}|+\cdots+|w_{\sigma_1}|$.  Suppose $w\in
W$ is an element of length $k$ and let $\sigma$ be the left descent
set.  Then by Proposition~\ref{prop:longestelement}, $w=w_{\sigma}v$
and $|w|=|w_{\sigma}|+|v|$.  Since $|v|<k$, the result follows by
induction. 
\end{proof}

\section{Rationality of the greedy normal form}\label{s:cox-rat}

Let $(W,S)$ be a Coxeter system, let $\cL\subset{ \cA}^*$ be its left
greedy normal form, and let $\chi=\chi_{\cL}$ be the characteristic
series for $\cL$.  In this section we shall prove that $\chi$ is
a rational power series (equivalently, that $\cL$ is a regular language).  Our
argument essentially constructs a finite state automaton whose
accepted language is $\cL$.   The automaton is a modification of the
``canonical automaton'' described in \cite{BB}, which is an
automaton whose accepted language is the set of all reduced
expressions for elements of $W$.  The language of reduced expressions
is canonical in that it does not depend on an ordering of the
generators $S$, but it does not have the uniqueness property -- i.e.,
a given element $w\in W$ can have many reduced expressions.  The 
greedy normal form $\cL$ is also canonical, but {\em does} have the
uniqueness property.   The trade-off that makes this possible seems to
be that the set of reduced expressions is defined over the alphabet
$S$, whereas $\cL$ is defined over the larger alphabet ${ \cA}$.

To describe our representation for $\chi$, we use the {\em geometric
representation} for $(W,S)$.  Details can be found in any standard
reference on Coxeter groups (\cite{Bourbaki, Davis-book, Humphreys, BB}).  Let $V$ be a
real vector space of dimension $|S|$, and let $(h_s)$ be a basis for
the dual space $V^*$.  For each $s\in S$, let $v_s\in V$ be the unique
vector such that $h_{s'}(v_s)=-\cos(\pi/m(s,s'))$ for all
$s'\in S$.  Then $\rho_s(v)=v-2 h(v)v_s$ defines a linear reflection
$\rho_s:V\rightarrow V$ across the hyperplane $\{v\;|\;
h_s(v)=0\}$.  Moreover, $s\mapsto\rho_s$ defines a faithful
representation $\rho:W\rightarrow GL(V)$, called the {\em geometric
representation} for $W$. To simplify notation, we identify
$W$ with its image in $GL(V)$.

Let $C\subset V$ be the cone 
\[C=\{v\;|\; h_s(v)\geq 0\;\mbox{for all}\;\; s\in S\}\]
and let $U\subset V$ be the union of $W$-translates of $C$:
\[U=\bigcup_{w\in W}wC.\]
By a theorem of Tits, $U$ is a convex cone, on which $W$ acts
discretely with fundamental domain $C$.  Moreover, a subgroup of $W$
is finite if and only if it is contained in the stabilizer of a point
in the interior of $U$.  In particular, the finite parabolic subgroups
$W_{\sigma}$ are precisely the stabilizers of points in the
fundamental chamber $C$ that lie in the interior of $U$.  The cone $U$
is usually refered to as the {\em   Tits cone}, the translates $wC$ as
{\em   chambers}, and $W$ as the {\em fundamental chamber}.   Since
$C$ is a fundamental domain, $w\leftrightarrow wC$ defines a
bijection between $W$ and the set of chambers.   

For each $s\in S$, we define the {\em fundamental hyperplane}
$H_s$ by $H_s=\{v\in U\;|\;h_s(v)=0\}$, and we let  
$\cH$ denote the collection of all $W$-translates of the fundamental
hyperplanes: 
\[\cH=\{wH_s\;|\; w\in W, s\in S\}.\]
Note that if $H=wH_s$ is any hyperplane in $\cH$, it is the fixed point set
of the reflection $r=wsw^{-1}\in W$, and separates $U$ into two
connected components. We shall refer to any $H\in \cH$ as a {\em
hyperplane} in $U$ and to (the closure of) a connected
component of $U-H$ as a {\em halfspace}.   Since hyperplanes do
not intersect the interiors of chambers, any intersection of
halfspaces is a union of chambers.   For each $H\in\cH$,
let $H^+$ (respectively, $H^-$) denote the half-space bounded by $H$
that contains (resp., does not contain) the fundamental chamber $C$.
We recall the following simple criterion for determining which
fundamental halfspaces contain a given chamber (see \cite{Bourbaki}). 

\begin{proposition} \label{prop:basic}
The chamber $wC$ is contained in the fundamental
  halfspace $H_s^-$ if and only if $|sw|<|w|$ (i.e., if and only if
  $s\in\desc(w)$).
\end{proposition}

For each spherical subset $\sigma\subset S$, let $\cH_{\sigma}$ denote
the collection of hyperplanes $\{wH_s\;|\;w\in W_{\sigma},\; s\in\sigma\}$,
and let $A_{\sigma}^+$ and $A_{\sigma}^-$ denote the subsets of $U$
defined by 
\[A_{\sigma}^+=\bigcap_{s\in\sigma}H_s^+\hspace{.5in}\mbox{and}\hspace{.5in}
A_{\sigma}^-=\bigcap_{s\in\sigma}H_s^-\]
(Figure~\ref{fig:Asigma}.)
\begin{figure}[ht]
\begin{center}
\psfrag{C}{$C$}
\psfrag{H1}{$H_{s}$}
\psfrag{H2}{$H_{t}$}
\psfrag{A12-}{$A_{st}^{-}$}
\psfrag{A12+}{$A_{st}^{+}$}
\includegraphics[scale = 1]{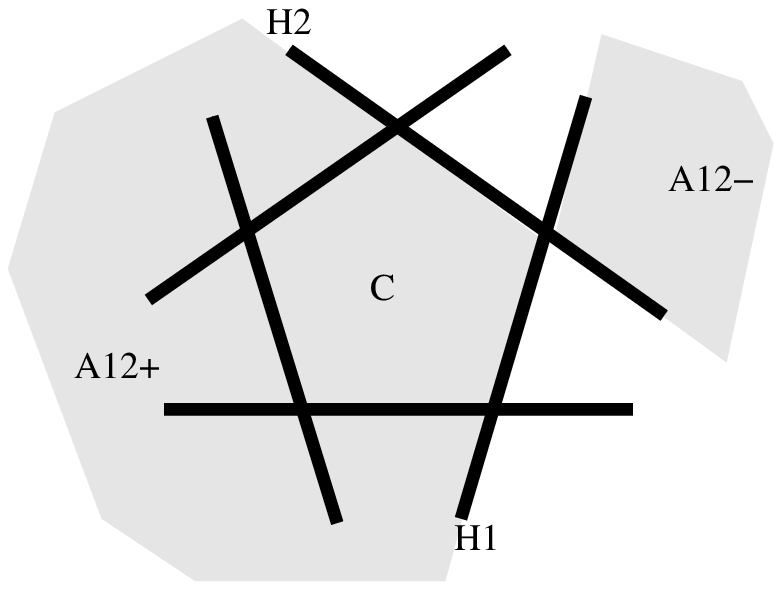}
\caption{\label{fig:Asigma}}
\end{center}
\end{figure}
In particular, the hyperplanes $\cH_{\sigma}$ are permuted by $W_{\sigma}$ and
$A_{\sigma}^+$ is a fundamental domain for the action of $W_{\sigma}$ on
this arrangement.  By definition of $w_{\sigma}$,
$|sw_{\sigma}|<|w_{\sigma}|$ for every $s\in\sigma$, hence by
Proposition~\ref{prop:basic} the chamber $w_{\sigma}C$  (and hence
$w_{\sigma}A_{\sigma}^+$) lies in $A_{\sigma}^-$.   

It also follows from Proposition~\ref{prop:basic} that
$A_{\sigma}^-\cap A_{\tau}^-=A_{\sigma\cup\tau}^-$ if
$\sigma\cup\tau$ is spherical (and is empty otherwise).  This leads to
a decomposition of the Tits cone cut out by the fundamental
hyperplanes; namely, we define $C_{\emptyset}=C$ and for
$\sigma\in { \cA}$ we
define $C_{\sigma}\subset U$ to be the subset
\[C_{\sigma}=\overline{A_{\sigma}^--\bigcup_{\sigma\subset \tau\in {
      \cA}}A_{\tau}^-}\]
(see Figure~\ref{fig:Csigma}).
\begin{figure}[ht]
\begin{center}
\psfrag{C}{$C_{\emptyset}$}
\psfrag{H1}{$H_{s}$}
\psfrag{H2}{$H_{t}$}
\psfrag{C1}{$C_{s}$}
\psfrag{C2}{$C_{t}$}
\psfrag{C12}{$C_{st}$}
\includegraphics[scale = 1]{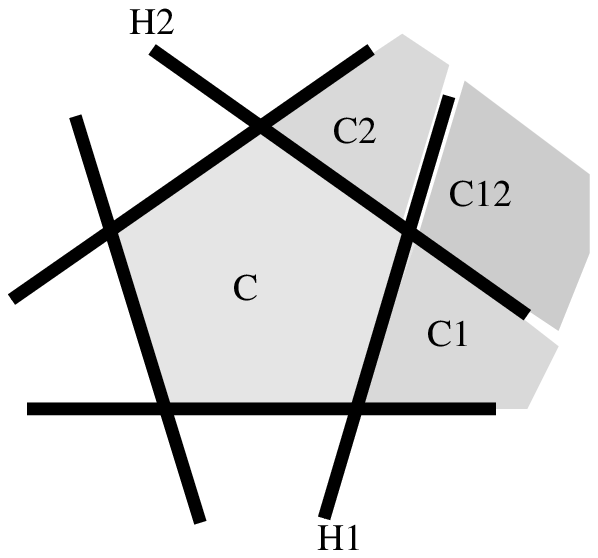}
\caption{\label{fig:Csigma}}
\end{center}
\end{figure}
Note that $C_{\sigma}$ can also be obtained by intersecting
$A_{\sigma}^-$ with the halfspaces $H_s^+$ where $s$ ranges over all
elements of $S$ such that $s\not\in\sigma$ and $\sigma\cup\{s\}$ is
spherical.  Thus, the $C_{\sigma}$'s are also intersections of
halfspaces, hence are unions of chambers.  More precisely, we have the
following decomposition. 

\begin{proposition} \label{prop:descent-test} The set $C_{\sigma}$ is the union of the
  chambers $wC$ such that $\desc(w)=\sigma$.
\end{proposition}

\begin{proof}{}
By Proposition~\ref{prop:basic}, $wC\subset H_s^-$ if and only if
$l(sw)<l(w)$, i.e., if and only if $s\in\desc(w)$.   Thus,
$A_{\sigma}^-$ is the union of those chambers $wC$ such that
$\desc(w)\supset\sigma$.  Since $C_{\sigma}$ is obtained from
$A_{\sigma}^-$ by removing the interior of every $A_{\tau}^-$ for
$\tau\supset\sigma$, the result follows.  
\end{proof}

Next we describe a finer decomposition (than the $C_{\sigma}$'s) of
$U$.  We define a partial order $\leq$ on $\cH$ by: $H_1\leq H_2$ if
and only if $H_1^+\subset H_2^+$.  In other words, $H_1<H_2$ if and
only if $H_1$ separates the hyperplane $H_2$ from the fundamental
chamber $C$. 
Let $\cH_{min}$ denote the set of minimal elements in the poset $\cH$.
In particular, each fundamental hyperplane $H_s$ is in $\cH_{min}$,
and more generally, for any spherical $\sigma$, the hyperplanes in
$\cH_{\sigma}$ are all in $\cH_{\min}$ (since they intersect the
fundamental chamber in the interior of $U$).
We recall some fundamental properties of the set
of minimal hyperplanes.  The first statement below is the main result
in \cite{BH}, the second is Proposition~2.2 in \cite{Casselman}. 

\begin{proposition}\hspace*{.1in}\label{prop:minprops}  
\begin{enumerate}
\item The set $\cH_{min}$ is finite.
\item If $s\in S$ and $H\in\cH_{\min}$, then either $sH\in\cH_{\min}$
  or $H_s\leq sH$.
\end{enumerate}
\end{proposition}

Let $\cR$ be the collection of regions of $U$ cut out by
$\cH_{\min}$.  Since each fundamental hyperplane is in $\cH_{\min}$,
the fundamental chamber $C$ is one of the regions in $\cR$.  More
generally, each region $R$ in $\cR$ is a union of chambers.
We are interested in how these regions behave under
translations by the elements $w_{\sigma}$.  Recall that for each
spherical subset $\sigma$, the subgroup $W_{\sigma}$ acts on $U$
preserving the collection of hyperplanes $\cH_{\sigma}$ and with
fundamental domain $A_{\sigma}^+$.  In particular, the translates
$wA_{\sigma}^+$ for $w\in W_{\sigma}$ are the pieces cut out by
$\cH_{\sigma}$ and since these hyperplanes are all minimal, these
pieces are further divided into regions in $\cR$.
We can now state the key technical observation for the construction of
our representation. 

\begin{lemma}  \label{lem:s-translate}
Suppose $\sigma\subset S$ is spherical, $w\in
  W_{\sigma}$, and $s\in S$ satisfies $|sw|>|w|$.  Then for every
  region $R\in\cR$ such that $R\subseteq wA_{\sigma}^+$ there exists a
  region $R'\in\cR$ such that $sR\subseteq R'\subseteq
  swA_{\sigma}^+$.  In particular, if $R\subset A_{\sigma}^+$, then
  there exists an $R'\in\cR$ such that $w_{\sigma}R\subseteq
  R'\subseteq w_{\sigma}A_{\sigma}^+=A_{\sigma}^-$. 
\end{lemma} 

\begin{proof}{}
Since $l(sw)>l(w)$, we have $wC\subseteq H_s^+$ (Figure~\ref{fig:automaton}).  
\begin{figure}[ht]
\begin{center}
\psfrag{C}{$C$}
\psfrag{wC}{$wC$}
\psfrag{Hs}{$H_{s}$}
\psfrag{Ht}{$H_{t}$}
\psfrag{Ast+}{$A_{st}^+$}
\psfrag{wAst+}{$wA_{st}^+$}
\psfrag{swAst+}{$swA_{st}^+$}
\psfrag{R}{$R$}
\psfrag{sR}{$sR$}
\psfrag{R'}{$R'$}
\psfrag{H}{$H$}
\psfrag{sH}{$sH$}
\includegraphics[scale = 1]{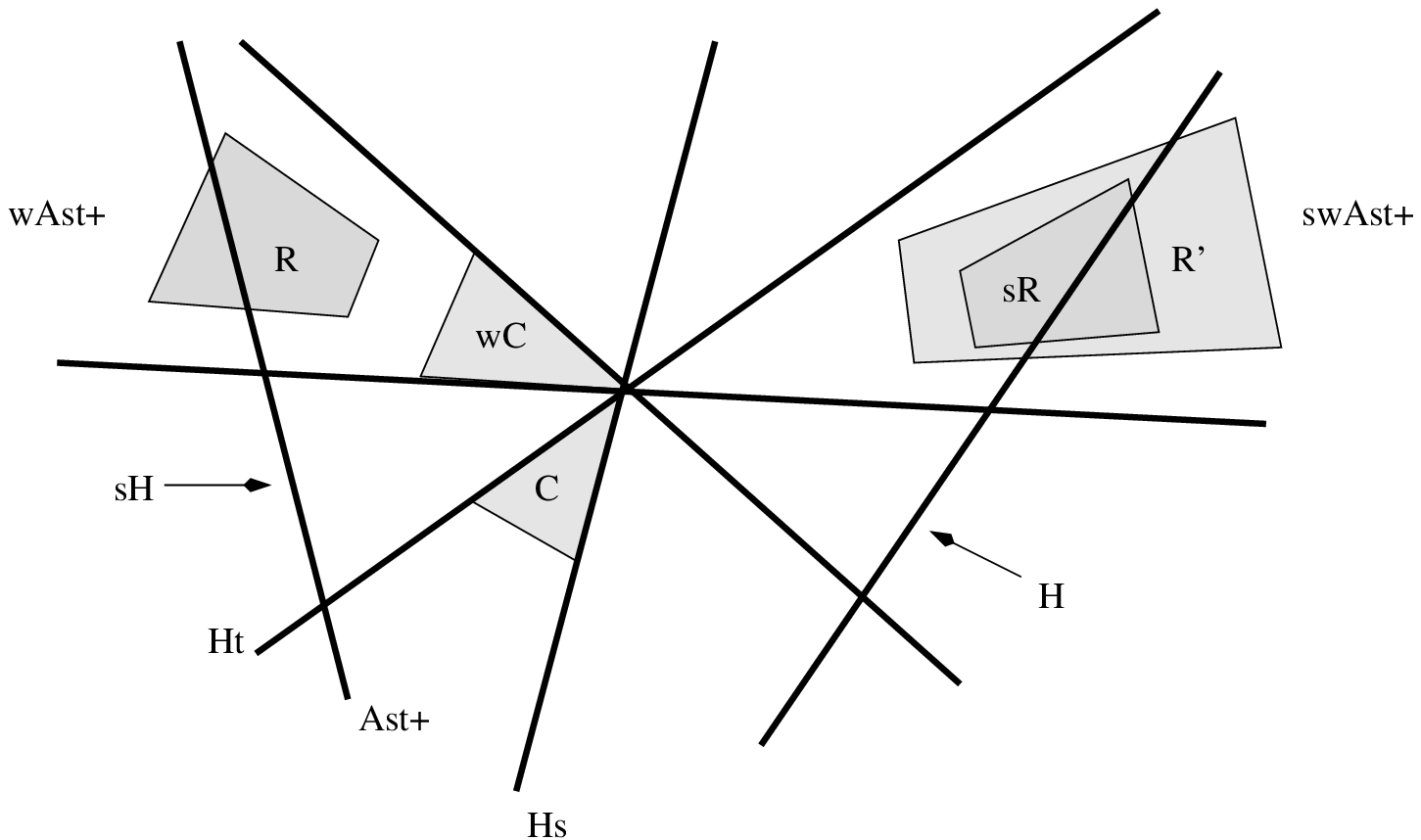}
\caption{\label{fig:automaton}}
\end{center}
\end{figure}
Since $wC$ and
$wA_{\sigma}^+$ lie on the same side of every hyperplane in
$\cH_{\sigma}$, we also know 
$R\subseteq wA_{\sigma}^+\subseteq H_s^+$.  Now suppose on the contrary
that $sR$ were {\em not} contained in some region $R'$.  This means
that (the interior of) $sR$ would have to be cut by some minimal
hyperplane $H\in\cH_{min}$.  Multiplying by $s$, we would then have
$sH$ intersecting the interior of the region $R$.  Since the
interior of $R$ is disjoint from every minimal hyperplane, this means
$sH$ cannot be minimal, hence by (2) of
Proposition~\ref{prop:minprops}, we must have $sH\geq H_s$ or,
equivalently, $H_s^+\subseteq (sH)^+$.  But this implies $R\subset
(sH)^+$ which is impossible since the hyperplane $sH$ intersects the
interior of the region $R$.  It follows that there must exist a region
$R'$ such that $sR\subseteq R'$.  Moreover, since $sR\subseteq
swA_{\sigma}^+$ and the latter is bounded by minimal hyperplanes , we
must have $R'\subseteq swA_{\sigma}^+$ as well. 

The final claim of the lemma follows by applying the first part of the
lemma iteratively to any reduced expression $s_k\ldots s_1$ for
$w_{\sigma}$.  
\end{proof}

We can now describe the inductive lemma we need to construct our
rational representation for $\chi_{\cL}$.  We set $W^R=\{w\in
W\;|\;wC\subset R\}$ and $\cL_R=\{\alpha\in\cL\;|\; \pi(\alpha)\in
W^R\}$.  Thus, $W^R$ is the set of group elements that move the
fundamental chamber into the region $R$, and $\cL_R$ is the set of
greedy representatives for those elements.  As a point of 
clarification, we note that the {\em length} of a word $\alpha$ is its
length as a string of letters.  (This is not the same as
$|\pi(\alpha)|$, the length of $\pi(\alpha)$ in $W$.)

\begin{lemma} \label{lem:induction} 
Suppose $\alpha$ is a word of length $k$ in $\cL_R$.  Then
$\sigma\alpha$ is a word of length $k+1$ in $\cL$ if and only if
$w_{\sigma}R\subset C_{\sigma}$.  Moreover, in this case, there exists
a unique region $R'\in\cR$ such that $w_{\sigma}R\subseteq R'$
(that is, such that $\sigma\alpha\in\cL_{R'}$).   
\end{lemma}

\begin{proof}{}
Let $w=\pi(\alpha)$.  Then $\sigma\alpha$ is greedy if and only if
$\sigma$ is the left descent set of $w_{\sigma}w$.  By
Proposition~\ref{prop:descent-test}, this is equivalent to
$w_{\sigma}wC\subseteq C_{\sigma}$.  Since $wC\subseteq R$ (because
$w\in\cL_R$) and $C_{\sigma}$ is a union of regions in $\cR$, this is
also equivalent to $w_{\sigma}R\subseteq C_{\sigma}$.  It remains to
show that $w_{\sigma}R$ is contained in a unique region $R'\in\cR$.
For this we note that $w_{\sigma}R\subseteq C_{\sigma}$ implies
$R\subseteq w_{\sigma}C_{\sigma}\subseteq
w_{\sigma}A_{\sigma}^-=A_{\sigma}^+$.  Thus, by Lemma~\ref{lem:s-translate},
there exists an $R'\in\cR$ such that $w_{\sigma}R\subseteq R'$.  
\end{proof}

Let $\chi\in\bbQ\<\<{ \cA}\>\>$ be the characteristic series for $\cL$ and
let $\chi_R$ denote the characteristic series for $\cL_{R}$.
Since the regions are disjoint, we have 
\[\chi=\sum_{R\in\cR}\chi_R.\]

Let $n=|\cR|$ and fix an ordering of $\cR$ (with the fundamental
chamber $R=C$ first in the order).  We can then index the entries of
vectors in $\bbQ\<\<{ \cA}\>\>^{n\times 1}$ by the set $\cR$ and the
entries of matrices in $\bbQ\<\<{ \cA}\>\>^{n\times   n}$ by the set
$\cR\times\cR$.   

\begin{theorem}\label{thm:cox-rat}
Let $Q=(Q_{R',R})\in \bbQ\<{ \cA}\>^{n\times n}$ be the matrix given by  
\[Q_{R',R}=\left\{\begin{array}{ll}
\sigma & \mbox{if $w_{\sigma}R\subset R'$ and $R'\subset B_{\sigma}$}\\
0 & \mbox{otherwise}\end{array}\right.,\]
and let $A\in\bbQ^{1\times n}$, $B\in \bbQ^{n\times 1}$ be the vectors  
\[A=\left[\begin{array}{cccc}1 & 1 & \cdots &
  1\end{array}\right]\hspace{.5in}
\mbox{and}\hspace{.5in}
B=\left[\begin{array}{cccc}1 & 0 & \cdots &
  0\end{array}\right]^T\]
Then $\chi=A(I-Q)^{-1}B$.  In particular, $\chi$ is rational (by
Theorem~\ref{thm:psrep}). 
\end{theorem}

\begin{proof}{}
For each $R\in\cR$, let $\cL_R^k$ denote the set of words in
$\cL_R$ of length $k$, and let $\chi_R^k$ be the corresponding
characteristic series. Let $X_k\in\bbQ\<\<{ \cA}\>\>^{n\times 1}$ be the
vector whose $R$th entry is $\chi_R^k$, and let $X$ be the sum 
\[X=X_1+X_2+\cdots.\]
In other words, $X$ is the vector whose $R$th entry is $\chi_R$, the
characteristic series for $\cL_R$.  Since there is
only one word of length $1$ (corresponding to the region $C$), we know
the initial vector $X_1=B$.  In general, by Lemma~\ref{lem:induction},
$X_k$ satisfies the recurrence  
\[X_{k+1}=QX_k.\] 
It follows that $X=(I+Q+Q^2+\cdots)B=(I-Q)^{-1}B$.  Since $\chi=\sum_R\chi_R=AX$, the result follows.
\end{proof}

\begin{remark}
In proving rationality, we have essentially constructed a finite state
automaton that recognizes the language $\cL$.  The state set is the
set $\cR\cup\{F\}$ where $F$ is a single fail state.  The start state
is $C$, and the transition function
$\mu:{ \cA}\times\cR\cup\{F\}\rightarrow\cR\cup\{F\}$ is defined by 
\[\mu(\sigma,R)=\left\{\begin{array}{ll}
R'&\mbox{if $R\neq F$ and $w_{\sigma}R\subset
      R'\subset B_{\sigma}$}\\
F & \mbox{otherwise}\end{array}\right.\]
\end{remark}

We illustrate Theorem~\ref{thm:cox-rat} with the following example.

\begin{example}\label{ex:b2}
Let $W$ be the affine Coxeter group $\tilde{B}_2$.  That is
$S=\{s_1,s_2,s_3\}$, and the Coxeter relations are 
\[(s_1s_2)^4=1\hspace{.5in}(s_2s_3)^4=1\hspace{.5in} (s_1s_3)^2=1.\]
There are 
$8$ minimal hyperplanes $H_1,H_2,\ldots, H_8$ cutting $U$ into the
$25$ regions $R_{\emptyset}=C$, $R_1$, $R_2$, $R_3$, $R_{23}$, $R_{27}$, $R_{14}$,
$R_{15}$, $R_{36}$, $R_{368}$, $R_{238}$, $R_{278}$, $R_{247}$, $R_{147}$, $R_{145}$,
$R_{156}$, $R_{356}$, $R_{3568}$, $R_{2368}$, $R_{2378}$, $R_{2478}$, $R_{1247}$,
$R_{1457}$, $R_{1456}$, $R_{1356}$ (see Figure~\ref{fig:b2}).  The regions
are indexed by the set of minimal hyperplanes that separate them from
the fundamental chamber.  The spherical subsets are $a=\{s_1\}$, $b=\{s_2\}$,
$c=\{s_3\}$, $x=\{s_2,s_3\}$, $y=\{s_1,s_3\}$, and $z=\{s_1,s_2\}$.
The longest elements $w_a$, $w_b$, and $w_c$ act on the figure by
reflecting across the hyperplane $H_1$, $H_2$, and $H_3$,
respectively.  The longest elements $w_{x}$, $w_{y}$, and $w_{z}$ act
on the figure by rotating $180^{\circ}$ about the points $H_2\cap
H_3$, $H_1\cap H_3$, and $H_1\cap H_2$, respectively. 

\begin{figure}[ht]
\begin{center}
\psfrag{C1}{$C_1$}
\psfrag{C2}{$C_2$}
\psfrag{C3}{$C_3$}
\psfrag{C12}{$C_{12}$}
\psfrag{C23}{$C_{23}$}
\psfrag{C13}{$C_{13}$}
\psfrag{h1}{$H_1$}
\psfrag{h2}{$H_2$}
\psfrag{h3}{$H_3$}
\psfrag{h4}{$H_4$}
\psfrag{h5}{$H_5$}
\psfrag{h6}{$H_6$}
\psfrag{h7}{$H_7$}
\psfrag{h8}{$H_8$}
\includegraphics[scale = 1]{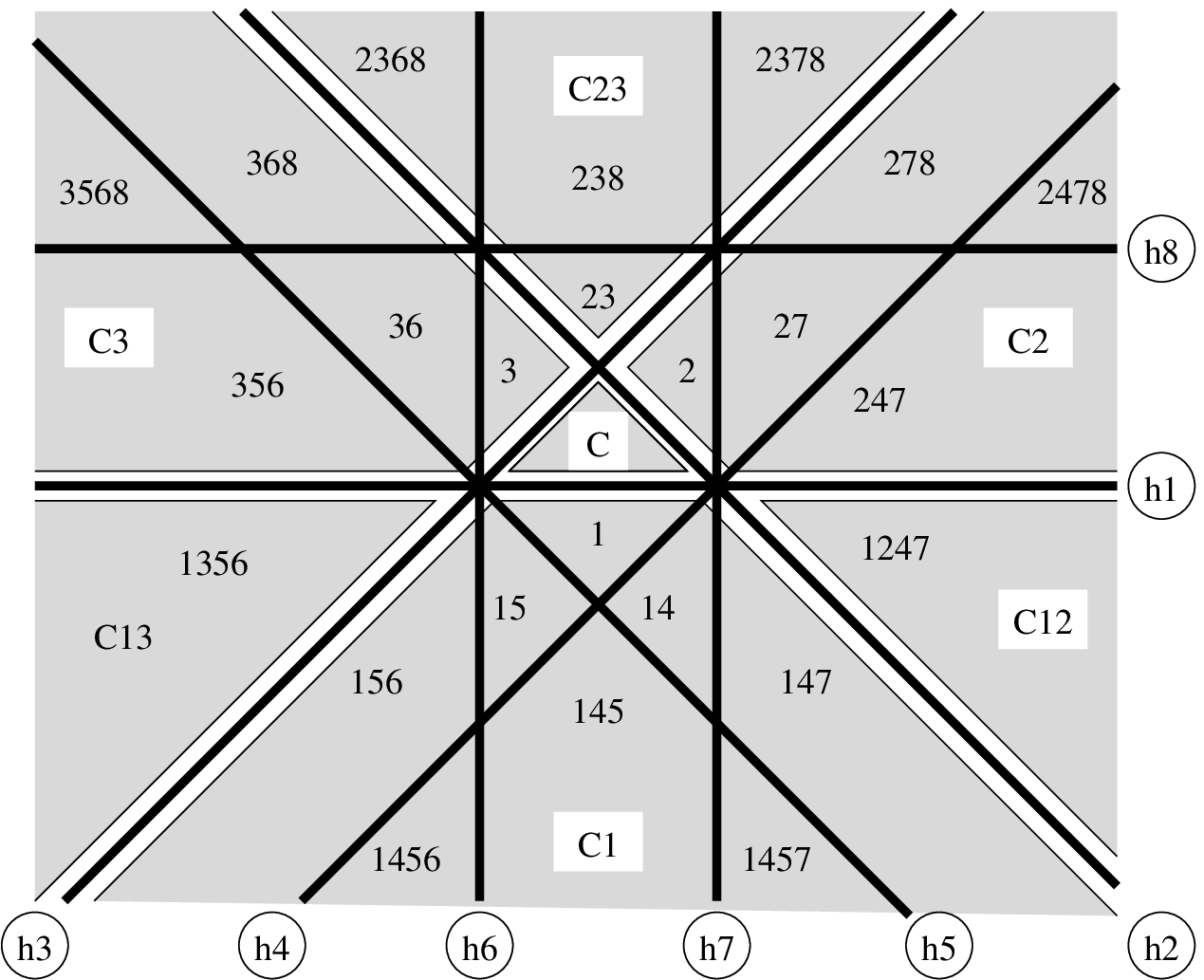}
\caption{\label{fig:b2}}
\end{center}
\end{figure}

With respect to this ordering of $\cR$, the matrix $Q$ of
Theorem~\ref{thm:cox-rat} is shown below. (To see where the entries of
this matrix came from, consider the region $R_{1456}$.  The
$w_x=s_2s_3s_2s_3$ translate of $R_{1456}$ lies in 
the region $R_{2378}$.  Since this region lies in $C_{23}$, the
$(R_{2378},R_{1456})$-entry of $Q$ is $x=\{s_2,s_3\}$.  On the other
hand, although the $w_c=s_3$ translate of $R_{1456}$ lands in
$R_{1356}$, the latter region is not contained in $C_3$, so the corresponding
entry of $Q$ is zero.  Both of these entries are indicated in bold in
the matrix.)

\[{\tiny\left[\begin{array}{ccccccccccccccccccccccccc}
0 & 0 & 0 & 0 & 0 & 0 & 0 & 0 & 0 & 0 & 0 & 0 & 0 & 0 & 0 & 0 & 0 & 0 & 0 & 0 & 0 & 0 & 0 & 0 & 0\\
a& 0 & 0 & 0 & 0 & 0 & 0 & 0 & 0 & 0 & 0 & 0 & 0 & 0 & 0 & 0 & 0 & 0 & 0 & 0 & 0 & 0 & 0 & 0 & 0\\
b& 0 & 0 & 0 & 0 & 0 & 0 & 0 & 0 & 0 & 0 & 0 & 0 & 0 & 0 & 0 & 0 & 0 & 0 & 0 & 0 & 0 & 0 & 0 & 0\\
c& 0 & 0 & 0 & 0 & 0 & 0 & 0 & 0 & 0 & 0 & 0 & 0 & 0 & 0 & 0 & 0 & 0 & 0 & 0 & 0 & 0 & 0 & 0 & 0\\
x& 0 & 0 & 0 & 0 & 0 & 0 & 0 & 0 & 0 & 0 & 0 & 0 & 0 & 0 & 0 & 0 & 0 & 0 & 0 & 0 & 0 & 0 & 0 & 0\\
0 & b& 0 & 0 & 0 & 0 & 0 & b& 0 & 0 & 0 & 0 & 0 & 0 & 0 & 0 & 0 & 0 & 0 & 0 & 0 & 0 & 0 & 0 & 0\\
0 & 0 & a& 0 & 0 & 0 & 0 & 0 & 0 & 0 & 0 & 0 & 0 & 0 & 0 & 0 & 0 & 0 & 0 & 0 & 0 & 0 & 0 & 0 & 0\\
0 & 0 & 0 & a& 0 & 0 & 0 & 0 & 0 & 0 & 0 & 0 & 0 & 0 & 0 & 0 & 0 & 0 & 0 & 0 & 0 & 0 & 0 & 0 & 0\\
0 & c& 0 & 0 & 0 & 0 & c& 0 & 0 & 0 & 0 & 0 & 0 & 0 & 0 & 0 & 0 & 0 & 0 & 0 & 0 & 0 & 0 & 0 & 0\\
0 & 0 & 0 & 0 & 0 & 0 & 0 & 0 & 0 & 0 & 0 & 0 & 0 & c& 0 & 0 & 0 & 0 & 0 & 0 & 0 & 0 & 0 & 0 & 0\\
0 & x& 0 & 0 & 0 & 0 & x& x& 0 & 0 & 0 & 0 & 0 & 0 & x& 0 & 0 & 0 & 0 & 0 & 0 & 0 & 0 & 0 & 0\\
0 & 0 & 0 & 0 & 0 & 0 & 0 & 0 & 0 & 0 & 0 & 0 & 0 & 0 & 0 & b& 0 & 0 & 0 & 0 & 0 & 0 & 0 & 0 & 0\\
0 & 0 & 0 & 0 & 0 & 0 & b& 0 & 0 & 0 & 0 & 0 & 0 & 0 & b& 0 & 0 & 0 & 0 & 0 & 0 & 0 & 0 & 0 & 0\\
0 & 0 & 0 & 0 & 0 & a& 0 & 0 & 0 & 0 & 0 & a& 0 & 0 & 0 & 0 & 0 & 0 & 0 & 0 & 0 & 0 & 0 & 0 & 0\\
0 & 0 & 0 & 0 & a& 0 & 0 & 0 & 0 & 0 & a& 0 & 0 & 0 & 0 & 0 & 0 & 0 & 0 & 0 & 0 & 0 & 0 & 0 & 0\\
0 & 0 & 0 & 0 & 0 & 0 & 0 & 0 & a& a& 0 & 0 & 0 & 0 & 0 & 0 & 0 & 0 & 0 & 0 & 0 & 0 & 0 & 0 & 0\\
0 & 0 & 0 & 0 & 0 & 0 & 0 & c& 0 & 0 & 0 & 0 & 0 & 0 & c& 0 & 0 & 0 & 0 & 0 & 0 & 0 & 0 & 0 & 0\\
0 & 0 & 0 & 0 & 0 & 0 & 0 & 0 & 0 & 0 & 0 & 0 & 0 & 0 & 0 & 0 & 0 & 0 & 0 & 0 & 0 & 0 & c& 0 & 0\\
0 & 0 & 0 & 0 & 0 & 0 & 0 & 0 & 0 & 0 & 0 & 0 & 0 & x& 0 & 0 & 0 & 0 & 0 & 0 & 0 & 0 & x& 0 & 0\\
0 & 0 & 0 & 0 & 0 & 0 & 0 & 0 & 0 & 0 & 0 & 0 & 0 & 0 & 0 & x& 0 & 0 &
0 & 0 & 0 & 0 & 0 & \underline{\mathbf x}& 0\\
0 & 0 & 0 & 0 & 0 & 0 & 0 & 0 & 0 & 0 & 0 & 0 & 0 & 0 & 0 & 0 & 0 & 0 & 0 & 0 & 0 & 0 & 0 & b& 0\\
z& 0 & 0 & z& 0 & 0 & 0 & 0 & z& z& 0 & 0 & 0 & 0 & 0 & 0 & z& z& 0 & 0 & 0 & 0 & 0 & 0 & 0\\
0 & 0 & 0 & 0 & 0 & 0 & 0 & 0 & 0 & 0 & 0 & 0 & 0 & 0 & 0 & 0 & 0 & 0 & 0 & a& 0 & 0 & 0 & 0 & 0\\
0 & 0 & 0 & 0 & 0 & 0 & 0 & 0 & 0 & 0 & 0 & 0 & 0 & 0 & 0 & 0 & 0 & 0 & a& 0 & 0 & 0 & 0 & 0 & 0\\
y& 0 & y& 0 & 0 & y& 0 & 0 & 0 & 0 & 0 & y& y& 0 & 0 & 0 & 0 & 0 & 0 &
0 & y& 0 & 0 & \underline{\mathbf 0} & 0\\
\end{array}\right].}\]

This gives
\begin{eqnarray*}
\chi &=&A(I-Q)^{-1}B\\
&=&A(I+Q+Q^2+\cdots)B\\
&=&1+(a+b+c+x+y+z)+(ab+ba+ac+ca+ax+xa+zc+yb)+\cdots.
\end{eqnarray*}
By formally inverting $I-Q$ using row reduction and inverses of
polynomials, one can obtain a rational expression for $\chi$, but it
will be a complicated nested expression in noncommuting polynomials
and their inverses. Instead, by replacing each letter with the single
variable $t$, we obtain the (rational) series:  
\[\bar{\chi}=1+6t+8t^2+\cdots=\frac{1+4t-3t^2+4t^3-2t^4}{1-2t+t^2}.\]   
The series $\bar{\chi}$ is the ``growth series'' for the language
$\cL(W,S)$; that is, it is the generating function for the sequence
$a(n)$ where $a(n)$ is the number of group elements whose left greedy
representations have length $n$ (in the monoid $\cA^*$). 
\end{example}  

\section{The anti-incidence matrix}\label{s:Eulerian}

We shall now assume that $W$ is a {\em
right-angled Coxeter group}; that is, the function $m$ in the
defining presentation satisfies $m(s,s')\in\{2,\infty\}$ for all
$s\neq s'$.  As before, we let $ N$ denote the nerve of $W$, $\cL$
denote the greedy normal form, and $\chi$ denote its characteristic
series.  Our goal for the
remainder of the paper is to show that for certain right-angled
Coxeter groups, the reciprocal $\chi^*$ exists and (up to sign)
coincides with $\chi$.  The proof of this formula reduces to showing
that a certain integer matrix is an involution.
This matrix can be defined for any poset $\cP$ and is reminiscent of the
zeta function (cf. \cite{Stanley}) except that instead of taking
values on pairs of related elements, it takes values on pairs that are
sufficiently far apart in the geometric realization of $\cP$.  For
this reason, we call it the {\em anti-incidence matrix} for $\cP$. 

Let $K$ be a finite simplicial complex with vertex set $V$.  For
convenience, we shall identify $K$ with its standard realization in
$\bbR^V$:  that is, we identify $V$ with the standard basis of
$\bbR^V$ and each simplex of $K$ with the convex hull of its
vertices. In keeping with the rest of the paper, we shall use
bold-face Greek letters $\bsigma,\btau,\brho$ for geometric simplices
and non-bold letters $\sigma,\tau,\rho$ for their respective vertex
sets.  The {\em dimension} of a simplex $\bsigma$, denoted by
$\dim\bsigma$, is $m-1$ where $m$ is the number of vertices of
$\bsigma$.   

We let $\cP(K)$
denote the set of simplices in $K$ (including the empty simplex
$\emptyset$, which has dimension $-1$), partially ordered by
inclusion.  We let $\bchi(K)$ denote the usual Euler characteristic:
\[\bchi(K)=\sum_{\bsigma\in\cP,\bsigma\neq\emptyset}(-1)^{\dim\bsigma}\]
and we let $\bar\bchi(K)$ denote the reduced Euler characteristic:
\[\bar{\bchi}(K)=\sum_{\bsigma\in\cP}(-1)^{\dim\bsigma}=\chi(K)-1.\]
(Again, we use the bold face symbol $\bchi$ to distinguish the Euler
characteristic from the characteristic series $\chi$ of a formal
language.)  

A {\em subcomplex} of $K$ is any subset $L\subseteq K$ that can be
expressed as a union of simplices in $\cP(K)$.  Given any subcomplex
$L$ of $K$, we let $\Vert(L)$ denote the set of 
vertices in $L$. A subcomplex $L$ is called {\em full} or {\em
  induced} if it is maximal among all subcomplexes with vertex set
$\Vert(L)$.  In other words, $L$ is a full subcomplex if it contains
every simplex $\bsigma$ of $K$ whose vertices are all in $L$.   

For any simplex $\bsigma$ in $K$, we recall the following standard
subcomplexes.     
\begin{enumerate}
\item[1.] $\St(\bsigma,K)$, the {\em star} of $\bsigma$ in $K$, is the union
  of all simplices $\btau$ such that   $\btau$ and   $\bsigma$ span a
  simplex in $K$.  
\item[2.] $\Lk(\bsigma,K)$, the {\em link} of $\bsigma$ in $K$, is the union
  of all simplices   $\btau$ in $\St(\bsigma,K)$ such that
  $\btau\cap\bsigma=\emptyset$.   
\end{enumerate}
Note that $\St(\bsigma,K)$ can be identified with the join 
$\bsigma\ast\Lk(\bsigma,K)$.   Our primary object of study in this
section is the following matrix, which is determined entirely by the
combinatorial structure of $K$.

\begin{definition}\label{def:aam}
Let $K$ be a simplicial complex and let $\cP=\cP(K)$.  Then the {\em
anti-incidence} matrix of $K$ is the $\cP\times\cP$ matrix $J$
defined by  
\[J_{\bsigma,\btau}=\left\{\begin{array}{ll}
(-1)^{\dim\bsigma} & \mbox{if $\btau\cap\St(\bsigma)=\emptyset$}\\
0 & \mbox{otherwise}\end{array}\right.\] 
\end{definition}

In light of the definition of this matrix, we single out some additional
subcomplexes.  First some more notation: given a subcomplex $L$, we
let $K-L$ denote the subcomplex of $K$ obtained by deleting all open
stars of vertices in $L$ (the {\em open star} of a simplex $\bsigma$ is
the union of all relative interiors of simplices that have nontrivial
intersection with $\bsigma$), and we let $\Nb(L,K)$ denote the union of
all (closed) stars of vertices in $L$.    Given any simplex
$\bsigma$ in $K$, we define the following two subcomplexes:
\begin{enumerate}
\item[3.] $B(\bsigma,K)=\Nb(\bsigma,K)$.  Equivalently, $B(\bsigma,K)$ is the
  union of all $\St(v,K)$ such that $v\in\Vert(\bsigma)$.  
\item[4.] $E(\bsigma,K)=K-\St(\bsigma,K)$.  Equivalently, $E(\bsigma,K)$ is
  the subcomplex spanned by all simplices $\btau$ such that
  $\btau\cap\St(\bsigma,K)=\emptyset$. 
Note that $B(\emptyset,K)=E(\emptyset,K)=\emptyset$.  When the
simplicial complex $K$ is understood, we will drop it from the
notation and refer to these subcomplexes simply 
as $\St(\bsigma)$, $\Lk(\bsigma)$, $B(\bsigma)$, and $E(\bsigma)$ (Figure~\ref{fig:ebsigma}).  
\end{enumerate}
\begin{figure}[ht]
\begin{center}
\psfrag{s}{$\bsigma$}
\psfrag{Lk}{$\Lk(\bsigma)$}
\psfrag{St}{$\St(\bsigma)$}
\psfrag{E}{$E(\bsigma)$}
\psfrag{B}{$B(\bsigma)$}
\includegraphics[scale = .5]{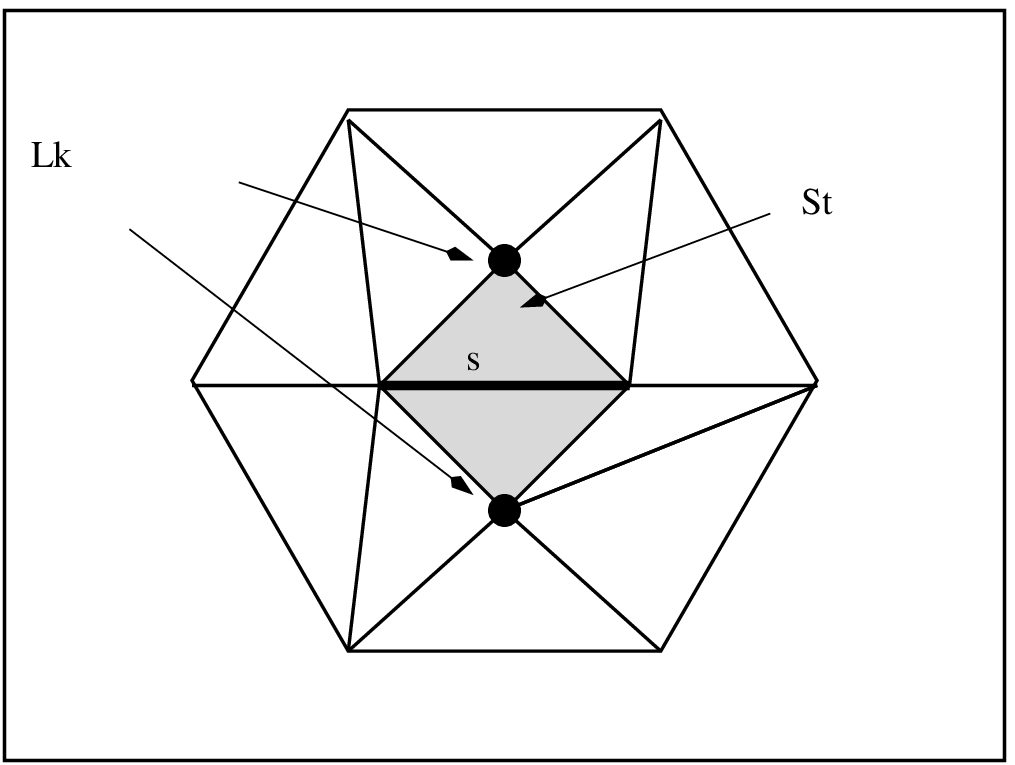}
\includegraphics[scale = .5]{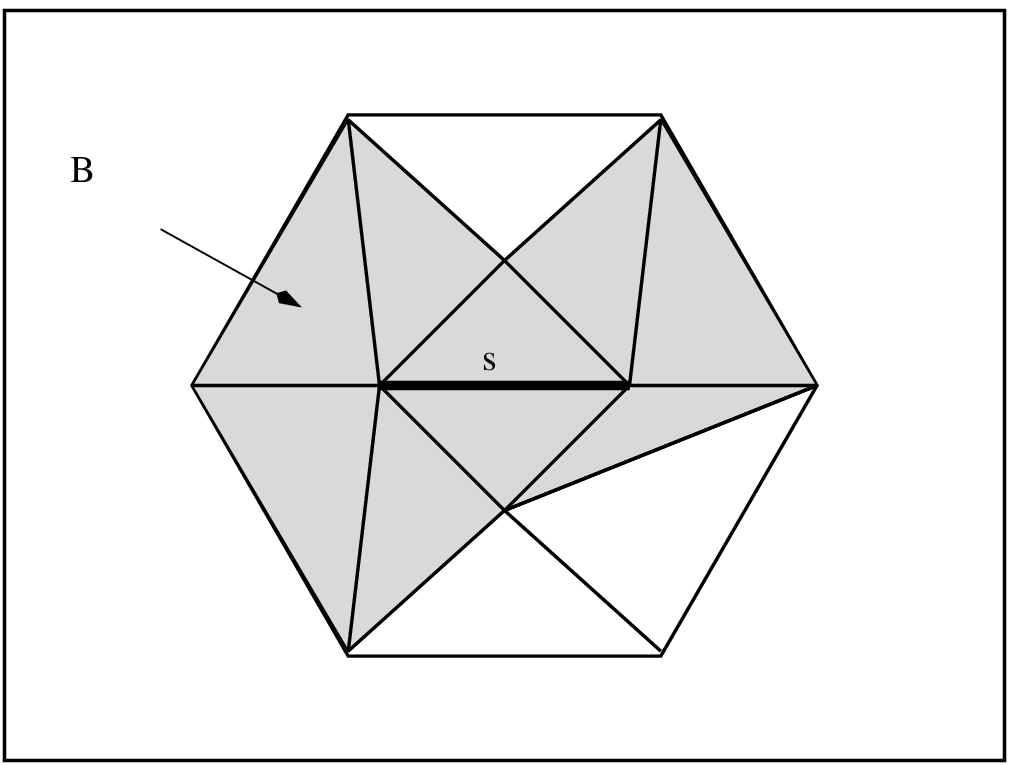}
\includegraphics[scale = .5]{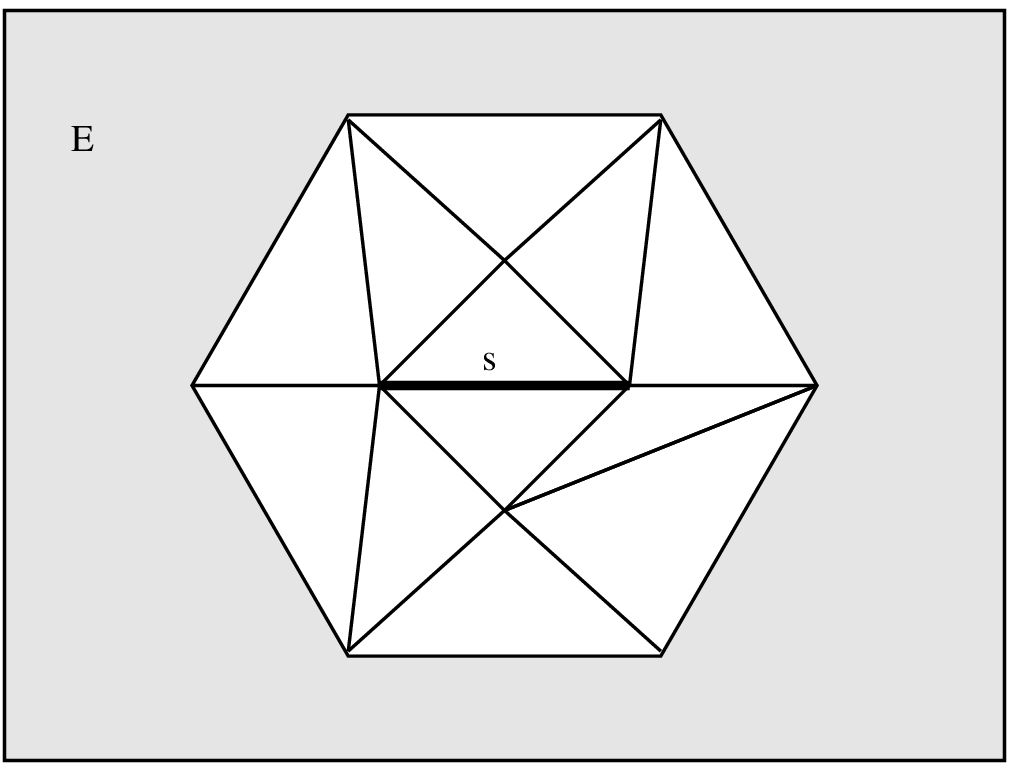}
\caption{\label{fig:ebsigma}}
\end{center}
\end{figure}

Of particular interest to us are simplicial complexes that are
determined by their vertices and edges in the following sense.

\begin{definition}
Let $K$ be a simplicial complex with vertex set $V$.  We say that two
vertices $v,w\in V$ are {\em adjacent} if they span an edge in $K$.  A
simplicial complex $K$ is a {\em flag complex} if any set of 
pairwise adjacent vertices spans a simplex in $K$.
\end{definition}

\begin{remark}\label{rem:sc-cox}
Recall that for a Coxeter group $(W,S)$, the nerve $N=N(W,S)$ can be regarded
as an abstract simplicial complex with vertex set $S$.  This means
that we can form a geometric simplicial complex (called the {\em
realization of $N$}, and denoted by $|N|$) as follows.  We identify
$S$ with the standard basis for $\bbR^S$ and we let $|N|\subset\bbR^S$
be the union of all convex hulls of subsets $\sigma\in N$.  It is a
standard result in topology that $\cP(|N|)=N$.  In the right-angled
case, a subset $\sigma\subset S$ generates a finite group if and only
if any pair of elements in $\sigma$ commute; thus the (geometric
realization of the) nerve of a right-angled Coxeter group is always a
flag complex.   Note that we have been using bold letters $\bsigma$,
$\btau$, etc., to denote geometric simplices and non-bold letters for
their respective vertex sets.  This is consistent with the rest of the
paper since any spherical subset $\sigma\in N$ spans a unique
geometric simplex $\bsigma\subseteq |N|$.  
\end{remark}

Any full subcomplex of a flag complex is a flag complex, and it is
easy to check from the definitions that if $K$ is a flag complex, then
$\Lk(\bsigma)$, $\St(\bsigma)$, and $E(\bsigma)$ are all full
subcomplexes (in fact $E(\bsigma)$ is a full subcomplex regardless of whether
$K$ is a flag complex).  Hence these subcomplexes are also flag.  The subcomplex
$B(\bsigma)$ need be neither full nor flag (take $K$ to be the
boundary complex of octahedron and $\bsigma$ to be one of the faces).
The next two propositions collect the key properties of flag complexes
that we shall need.

\begin{proposition}\label{prop:flagstar}
Let $K$ be a flag simplicial complex, and let $\bsigma$ be a nonempty
simplex in $K$.  Then 
\[\St(\bsigma)=\bigcap_{v\in\Vert(\bsigma)}\St(v).\]
\end{proposition}

\begin{proof}{}
The inclusion $\St(\bsigma)\subseteq\bigcap_{v\in\Vert(\bsigma)}\St(v)$
is clear.  We prove the opposite inclusion by induction on
$\dim\bsigma$.  The case $\dim\bsigma=0$ is clear.  If $\dim\bsigma>0$,
then $\bsigma$ can be written as the span of some vertex $v_0$ and some
simplex $\btau$ such that $\dim\btau<\dim\bsigma$.  By induction, we have 
\[\St(\btau)=\bigcap_{v\in\Vert(\btau)}\St(v)\]
so it suffices to prove $\St(v_0)\cap\St(\btau)\subseteq\St(\bsigma)$.
Suppose that $\brho$ is a simplex in $\St(v_0)\cap\St(\btau)$.  Then
$\brho$ and $v_0$ span a simplex in $K$, and $\brho$ and $\btau$ span a
simplex in $K$.  Since $v_0$ and $\btau$ also span a simplex, we
know that $\{v_0\}\cup\Vert(\btau)\cup\Vert(\brho)$ is a set of pairwise
adjacent vertices.  Since $K$ is a flag complex, this set of vertices
must span a simplex $\tilde{\brho}$.  But since $\bsigma$ and $\brho$
also span the simplex $\tilde{\brho}$, we have
$\brho\subseteq\St(\bsigma)$, which completes the proof.
\end{proof}  
 
\begin{proposition}\label{prop:ecapstar}
Let $K$ be a flag complex and let $\bsigma$ and $\btau$ be any
two simplices of $K$.  Set $\btau'=\btau\cap E(\bsigma)$ and
$\bsigma'=\bsigma-(\bsigma\cap\btau)$.   Then  
\[E(\bsigma)\cap\St(\btau)=
\left\{\begin{array}{ll}
\btau'\ast(\Lk(\btau)\cap E(\bsigma)) & \mbox{if $\btau'\neq\emptyset$}\\
E(\bsigma',\Lk(\btau)) & \mbox{if $\btau'=\emptyset$}
\end{array}\right..\]
\end{proposition}

\begin{proof}{}
Since $\St(\btau)=\btau\ast\Lk(\btau)$ and $E(\bsigma)$ is a full subcomplex, we have
\[E(\bsigma)\cap\St(\btau)=(\btau\cap E(\bsigma))\ast(\Lk(\btau)\cap
E(\bsigma))=\btau'\ast(\Lk(\btau)\cap E(\bsigma))\]
(Figure~\ref{fig:eb1}).
\begin{figure}[ht]
\begin{center}
\psfrag{s}{$\bsigma$}
\psfrag{t}{$\btau$}
\psfrag{t'}{$\btau'$}
\psfrag{Lk}{$Lk(\btau)$}
\psfrag{St}{$\St(\btau)$}
\psfrag{E}{$E(\bsigma)$}
\psfrag{LkcapE}{$\Lk(\btau)\cap E(\bsigma)$}
\psfrag{EcapSt}{$E(\bsigma)\cap\St(\btau)$}
\includegraphics[scale = .8]{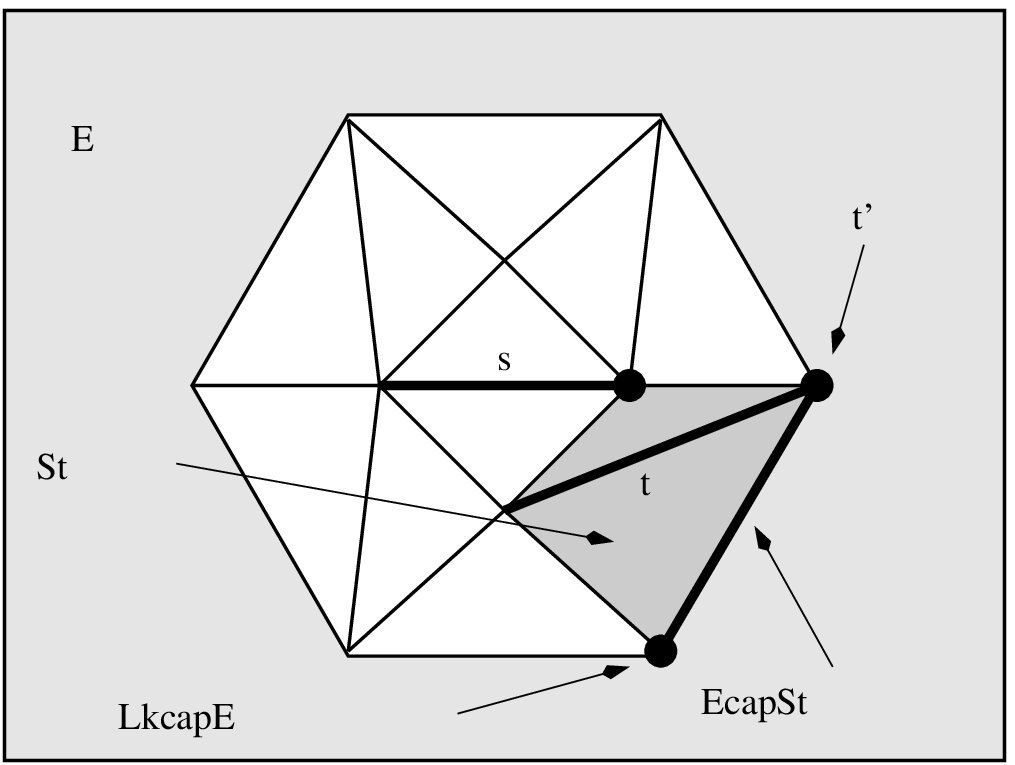}
\caption{\label{fig:eb1}}
\end{center}
\end{figure}
It suffices, then, to show that if $\btau'=\emptyset$ then  
\begin{eqnarray}\label{eq:lke}
\Lk(\btau)\cap E(\bsigma)=E(\bsigma',\Lk(\btau)).
\end{eqnarray}
Suppose $\btau'=\emptyset$.  Then $\btau\subseteq\St(\bsigma)$, so $\btau$
and $\bsigma$ span a simplex in $K$.  By definition of $\bsigma'$, this
simplex can be written as the join $\btau\ast\bsigma'$.
(Figure~\ref{fig:eb2} gives an illustration for the case where $\btau$
is a face of $\bsigma$, in which case $\btau$ and $\bsigma'$ span the
simplex $\bsigma$.)
\begin{figure}[ht]
\begin{center}
\psfrag{s}{$\bsigma$}
\psfrag{t}{$\btau$}
\psfrag{s'}{$\bsigma'$}
\psfrag{StLk}{$\St(\bsigma',\Lk(\btau))$}
\psfrag{St}{$\St(\btau)$}
\psfrag{E}{$E(\bsigma)$}
\psfrag{EcapSt}{$E(\bsigma)\cap\St(\btau)$}
\includegraphics[scale = .8]{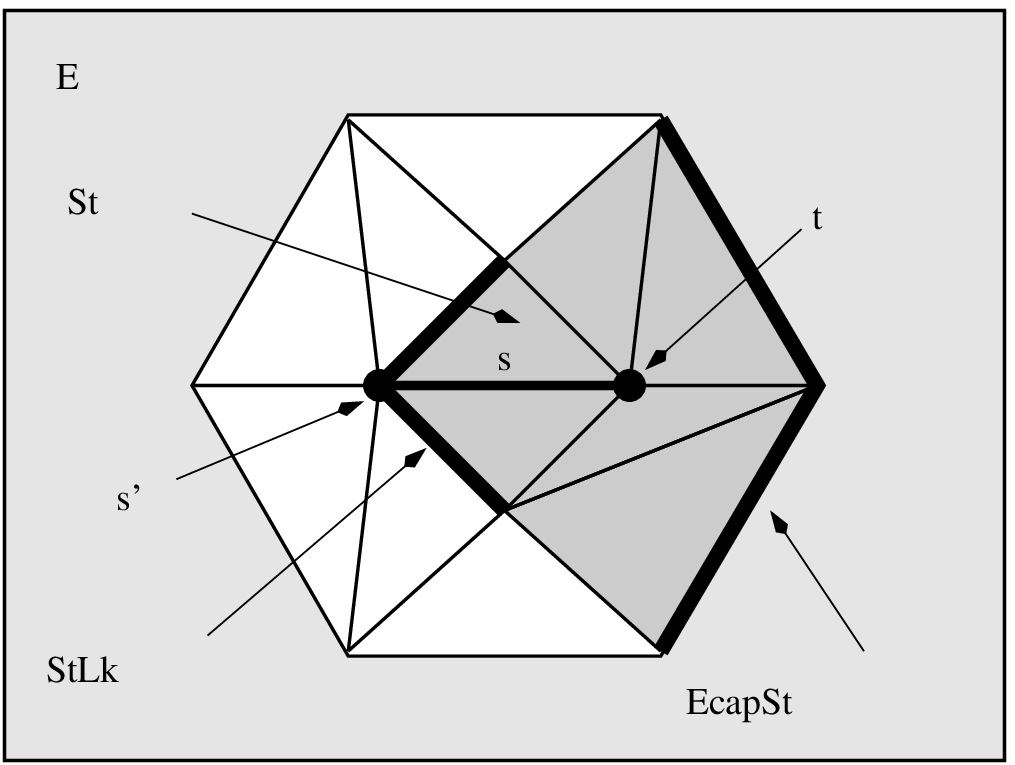}
\caption{\label{fig:eb2}}
\end{center}
\end{figure}
To prove (\ref{eq:lke}), we suppose $\brho$ is a simplex in $\Lk(\btau)$.  
We need to show that 
$\brho\cap\St(\bsigma)=\emptyset$ if and only if
$\brho\cap\St(\bsigma')=\emptyset$.  Since
$\St(\bsigma)\subseteq\St(\bsigma')$, one implication ($\Leftarrow$) is
obvious.  For the other direction, suppose
$\brho\cap\St(\bsigma')\neq\emptyset$, and pick a vertex $v$ in 
$\brho\cap\St(\bsigma')$.  Then $v$ and $\btau$ span a simplex in $K$ as
do $v$ and $\bsigma'$.  But since $\btau$ and $\bsigma'$ span a simplex
and $K$ is a flag complex, we know that $v$, $\btau$, and $\bsigma'$
span a simplex.  But $\bsigma$ must be a face of this simplex (since
it's a face of $\btau\ast\bsigma'$).  It follows that $v$ and $\bsigma$
span a simplex, so $v\in\St(\bsigma)$.  Hence
$\brho\cap\St(\bsigma)\neq\emptyset$, and this completes the proof.
\end{proof}

Surprisingly, the anti-incidence matrix turns out to be an
involution when $K$ is a sphere, and this fact lies at the heart of
our reciprocity formulas.   In fact, $K$ need only resemble a
sphere up to Euler characteristics in the following sense.

\begin{definition}
A simplicial complex $K$ is an {\em Eulerian sphere} of dimension $k$,
if $\bar{\bchi}(K)=(-1)^k$ and for every $\bsigma\in K$,
$\bar{\bchi}(\Lk(\bsigma))=(-1)^{k-1-\dim\bsigma}$.
If, in addition, $K$ is a flag complex, we shall call it an {\em
  Eulerian flag complex}.
\end{definition}

\begin{remark}
If $K$ is an Eulerian sphere, then so is the link of any simplex in
$K$.
\end{remark}

\begin{remark}
A simplicial complex $K$ is an Eulerian sphere if and only if the
poset $\cP(K)$ (with a single maximal element $\hat{1}$ added) is an
{\em Eulerian poset} in the sense of \cite{Stanley}.
\end{remark}

\begin{example}
Obviously if $K$ is a triangulation of a sphere, then it is Eulerian.
For other examples, one can take $K$ to be two disjoint circles (or
their suspension). 
\end{example}

The connection between Coxeter groups whose nerve is Eulerian
and reciprocity formulas for growth series of those Coxeter
groups is well-known (\cite{CD}) when the series is defined over 
a {\em commutative} ring.  But the matrix $J$ seems to be
new, and not only yields new proofs of some of these known formulas, but
also provides the key to reciprocity formulas for noncommutative power
series.  Not surprisingly, the relevant facts about $J$ are related to
various Euler characteristics of $K$ and its subcomplexes.

\begin{lemma}\label{lem:BEchi1}
Let $K$ be a flag complex and let $\bsigma$ be a nonempty simplex in
$K$.  Then $\bar{\bchi}(B(\bsigma))=0$.   If in addition $K$ is Eulerian,
then $\bar{\bchi}(E(\bsigma))=0$.
\end{lemma}

\begin{proof}{}
Applying inclusion-exclusion and Proposition~\ref{prop:flagstar} to 
\[B(\bsigma)=\bigcup_{v\in\Vert(\bsigma)}\St(v)\]
we obtain
\[\bar{\bchi}(B(\bsigma))=\sum_{\btau\subseteq\bsigma,\btau\neq\emptyset}
(-1)^{\dim\btau}\bar{\bchi}(\St(\btau)).\]
Since $\bar{\bchi}(\St(\btau))=0$  ($\St(\btau)$ is contractible), this
simplifies to $0$.   

The proof that $\bar{\bchi}(E(\bsigma))=0$ follows from Proposition~3.14.5 of
\cite{Stanley} which says that up to sign, the reduced Euler
characteristic of a subcomplex $Q$ of an Eulerian sphere $K$ coincides
with the reduced Euler characteristic of the (topological) complement
of $Q$ in $K$ (we denote this open subset of $K$ by $K\setminus Q$ to
distinguish it from the closed subcomplex $K-Q$).  In the
present case, we take $Q$ to be $\St(\bsigma)$, which has reduced Euler
characteristic $0$, hence so does $K\setminus\St(\bsigma)$.  We claim that
$K\setminus\St(\bsigma)$ has a piecewise 
linear deformation retract onto $E(\bsigma)=K-\St(\bsigma)$.  We define this retract
simplex by simplex.  If $\btau\subset E(\bsigma)$, then the retract
restricted to $\btau$ is the identity map.  If $\btau\not\subset
E(\bsigma)$, then  
$\btau$ must have a nontrivial intersection with $\St(\bsigma)$.  In
fact, since $\St(\bsigma)$ is a full subcomplex of $K$, this
intersection must be a face of $\btau$.  It follows that $\btau$ can be
written as the join $\btau_1\ast\btau_2$ where
$\btau_1=\btau\cap\St(\bsigma)$ and $\btau_2=\btau\cap E(\bsigma)$.  The
intersection of $\btau$ with $K\setminus \St(\bsigma)$ is $\btau\setminus
\btau_1$, and the retract restricted to $\btau\setminus\btau_1$ is the
standard retract onto the face $\btau_2$ along the join lines.  These
retracts on simplices glue together to give a well-defined deformation
retract from $K\setminus\St(\bsigma)$ onto $E(\bsigma)$.  It follows that
$E(\bsigma)$ also has reduced Euler characteristic $0$, which completes
the proof.     
\end{proof}

In fact, when $K$ is Eulerian we can also compute Euler
characteristics for the subcomplexes $E(\bsigma)\cap B(\btau)$.  It
turns out that these are precisely the numbers that appear when we
square the anti-incidence matrix.  First we consider the Euler
characteristics of the intersections $E(\bsigma)\cap\St(\btau)$ when $K$
is Eulerian.

\begin{lemma} 
Let $K$ be an Eulerian flag complex, and let $\bsigma$
and $\btau$ be any two nonempty simplices in $K$.  Then
\[\bar{\bchi}(E(\bsigma)\cap\St(\btau))=\left\{\begin{array}{ll}
-1 & \mbox{if $\bsigma\subseteq\btau$}\\
0 & \mbox{if $\bsigma\not\subseteq\btau$}\\
\end{array}\right.\]
\end{lemma}

\begin{proof}{}
If $\btau\subseteq\St(\bsigma)$, then by
Proposition~\ref{prop:ecapstar}, $E(\bsigma)\cap\St(\btau)$ is
$\emptyset$ when $\bsigma\subseteq\btau$ (hence has reduced Euler
characteristic $-1$), and is of the form $E(\bsigma',\Lk(\btau))$ with
$\bsigma'\neq\emptyset$ otherwise.  Since $\Lk(\btau)$ is also an Eulerian
flag complex, it follows from Lemma~\ref{lem:BEchi1} that in the
latter case, the reduced Euler characteristic is $0$.  

If $\btau\not\subseteq\St(\bsigma)$, then $\btau'=\btau\cap
E(\bsigma)\neq\emptyset$ so by
Proposition~\ref{prop:ecapstar}, $E(\bsigma)\cap\St(\btau)$ is the join
of $\btau'$ with some other subcomplex.  Such a join is always
contractible, hence has reduced Euler characteristic $0$. 
\end{proof}

We can now use inclusion-exclusion and the previous lemma to compute
the Euler characteristics of $E(\bsigma)\cap B(\btau)$. 

\begin{lemma} 
Let $K$ be a an Eulerian flag complex, and let $\bsigma$
and $\btau$ be any two nonempty simplices in $K$.  Then
$\bar{\bchi}(E(\bsigma)\cap B(\btau))$ is equal to $0$ if
$\bsigma\neq\btau$ and is equal to $(-1)^{1+\dim\bsigma}$ if
$\bsigma=\btau$. 
\end{lemma}

\begin{proof}{}
We first calculate the Euler characteristic of the union 
\[E(\bsigma)\cup
B(\btau)=\bigcup_{v\in\Vert(\bsigma)}(E(\bsigma)\cup\St(v)).\]
Applying inclusion-exclusion and Proposition~\ref{prop:flagstar}, we
have 
\begin{eqnarray}\label{eq:union}
\bar{\bchi}(E(\bsigma)\cup B(\btau))
=\sum_{\brho\subseteq\btau,\brho\neq\emptyset}(-1)^{\dim\brho} 
\bar{\bchi}(E(\bsigma)\cup\St(\brho))
\end{eqnarray}
Using the previous lemma and the fact that
\begin{eqnarray*}
\bar{\bchi}(E(\bsigma)\cup\St(\brho))& =
&\bar{\bchi}(E(\bsigma))+\bar{\bchi}(\St(\brho))-\bar{\bchi}(E(\bsigma)\cap\St(\brho))\\
& = & -\bchi(E(\bsigma)\cap\St(\brho)),
\end{eqnarray*}
we have 
\[\bar{\bchi}(E(\bsigma)\cup\St(\brho))=\left\{\begin{array}{ll}
0 & \mbox{if $\bsigma\not\subseteq\brho$}\\
1 & \mbox{if $\bsigma\subseteq\brho$}\end{array}\right.\]
Substituting into (\ref{eq:union}) gives 
\begin{eqnarray*}
\bar{\bchi}(E(\bsigma)\cup B(\btau))
& = & \sum_{\bsigma\subseteq\brho\subseteq\btau}(-1)^{\dim\brho}.
\end{eqnarray*}
This evaluates to $0$ if $\bsigma\neq\btau$ and evaluates to
$(-1)^{\dim\bsigma}$ if $\bsigma=\btau$.   
To complete the proof, we use the identity
\begin{eqnarray*}
\bar{\bchi}(E(\bsigma)\cap B(\btau))& =
&\bar{\bchi}(E(\bsigma))+\bar{\bchi}(B(\btau))-\bar{\bchi}(E(\bsigma)\cup
B(\btau))\\
& = & -\bchi(E(\bsigma)\cup B(\btau)),
\end{eqnarray*}
\end{proof}

We now come to the main result of the section.

\begin{theorem}\label{thm:j}
Let $K$ be a flag complex and let $J=J(K)$ be the anti-incidence
matrix for $K$.  Then the columns of $J$ all have sum $\bar{\bchi}(K)$.
If, in addition, $K$ is Eulerian then $J^2=\Id$.
\end{theorem}

\begin{proof}{}
The column of $J$ corresponding to $\btau$ has sum
\[\sum_{\btau\cap\St(\bsigma)=\emptyset}(-1)^{\dim\bsigma} = 
\sum_{\bsigma\subset
  K}(-1)^{\dim\bsigma}-\sum_{\btau\cap\St(\bsigma)\neq\emptyset}(-1)^{\dim\bsigma}.\]
This simplifies to $\bar{\bchi}(K)-\bar{\bchi}(B(\btau))$ if
$\btau\neq\emptyset$ and to $\bar{\bchi}(K)-0$ if $\btau=\emptyset$.  Thus,
by Lemma~\ref{lem:BEchi1}, we obtain $\bar{\bchi}(K)$ in either case.

Now assume that $K$ is Eulerian.  The $\bsigma,\btau$-entry of $J^2$ is 
\[\sum_{\brho\cap\St(\bsigma)=\emptyset, \btau\cap\St(\brho)=\emptyset}
(-1)^{\dim\bsigma}(-1)^{\dim\brho}.\]
If $\bsigma=\btau=\emptyset$, this entry simplifies to $1$.  If
$\bsigma=\emptyset\neq\btau$, this entry simplifies to $0$.  If
$\bsigma\neq\emptyset=\btau$, this entry simplifies to 
\[(-1)^{\dim\bsigma}\sum_{\brho\cap\St(\bsigma)=\emptyset}(-1)^{\dim\brho}
=(-1)^{\dim\bsigma}\bar{\bchi}(E(\bsigma))\]
which is equal to $0$ by Lemma~\ref{lem:BEchi1}. 
If $\bsigma$ and $\btau$ are both nontrivial, this entry simplifies to 
\begin{eqnarray*}
(-1)^{\dim\bsigma}\left(\sum_{\brho\cap\St(\bsigma)=\emptyset}(-1)^{\dim\brho} -
\sum_{\brho\cap\St(\bsigma)=\emptyset, \btau\cap\St(\brho)\neq\emptyset}(-1)^{\dim\brho
}\right)\\
=(-1)^{\dim\bsigma}(\bar{\bchi}(E(\bsigma))-\bar{\bchi}(E(\bsigma)\cap
B(\btau)))\\
=(-1)^{\dim\bsigma}\bar{\bchi}(E(\bsigma)\cap B(\btau)),\\
\end{eqnarray*}
which by Lemma~\ref{lem:BEchi1} is equal to $0$ when $\bsigma\neq\btau$ and is 
equal to $1$ when $\bsigma=\btau$.  It follows that $J^2$ is the
identity matrix.   
\end{proof}

\section{Reciprocity for right-angled Coxeter groups}\label{s:cox-rec}

We now return to the proof of our reciprocity formula and combine the
results of the previous two sections.  An important observation 
that makes the series $\chi_{\cL}$ for right-angled Coxeter groups
much more manageable than for arbitrary Coxeter groups is that the
regions of the Tits cone cut out by minimal hyperplanes are actually
in bijection with the nerve.  More precisely, we have the following
(with notation as in Section~\ref{s:cox-rat}). 

\begin{proposition}
Assume that $W$ is a right-angled Coxeter group.  Then the set 
$\cH_{min}$ of minimal hyperplanes is precisely the set of fundamental
hyperplanes.  Moreover, the correspondence $\sigma\mapsto C_{\sigma}$
defines a bijection from the nerve $N$ to the set of regions $\cR$ cut
out by the minimal hyperplanes.
\end{proposition}

\begin{proof}{}
For a right-angled Coxeter group, the regions $C_{\sigma}$ are
precisely those cut out by the fundamental hyperplanes, hence the
second statement follows immediately from the first.  To show that the
fundamental hyperplanes are the only minimal ones, suppose $H$ is a
minimal hyperplane, and let $wC$ be a chamber having $H$ as wall
(i.e., $H\cap wC$ is a codimension-one face of $wC$).  Assume that
$wC$ is chosen so that $w$ has minimal length.  Since $w^{-1}H$ is a
wall of the fundamental chamber, it is a fundamental hyperplane, say
$H_s$.   If $|w|=0$, we're done since $wC=C$, so $H$ is a fundamental
hyperplane.  Otherwise, there exists an $s'\in S$ such that $ws'$ is
shorter than $w$.  Let $H'$ be the hyperplane fixed by the reflection
$r=ws'w^{-1}$.  Then $H$ and $H'$ are $w$ translates of the
fundamental hyperplanes $H_s$ and $H_{s'}$, hence are either
perpendicular or parallel.  If they were parallel, then $H'$ would
separate $H$ from the fundamental chamber, contradicting the
minimality of $H$.  On the other hand, if $H$ and $H'$ were
perpendicular, then the chamber $ws'C$ would still intersect $H$ in a
codimension-one face, contradicting the minimality of the length of
$w$.  Hence $wC$ must, in fact, be the fundamental chamber and $H$ a
fundamental hyperplane.
\end{proof}  

Since $\cR=\{C_{\sigma}\;|\; \sigma\in  N\}$, the recurrence described
in Theorem~\ref{thm:cox-rat} has a simpler description, depending only on the
combinatorics of the nerve.  Recall that the nerve $N$ can be regarded
as an abstract simplicial complex with geometric realization $|N|$.
As in Remark~\ref{rem:sc-cox}, for any abstract simplex $\sigma\in N$,
we let $\bsigma$ denote the corresponding geometric simplex in $|N|$. 

Replacing $R$ with $C_{\tau}$ and using the fact that $R'\subset
C_{\sigma}$ implies $R'=C_{\sigma}$, we obtain the following
simplification of Lemma~\ref{lem:induction}. 
 
\begin{proposition}\label{prop:star}
For $\sigma,\tau\in  N$, $w_{\sigma}C_{\tau}\subseteq C_{\sigma}$ if
and only if $\St(\bsigma)\cap\btau=\emptyset$.
\end{proposition}

\begin{proof}{}
Note that for a right-angled Coxeter group, $\sigma\in  N$
implies that $w_{\sigma}=\prod_{s\in\sigma}s$ (the order of
multiplication doesn't matter).   First we show that the condition
$\St(\bsigma)\cap\btau=\emptyset$ is equivalent to the statement
$\desc(w_{\sigma}w_{\tau})=\sigma$.  Indeed, $s\in
\St(\bsigma)\cap\btau$ implies that either (1)
$s\not\in\sigma$ and $\{s\}\cup\sigma\in  N$ in which case
$w_{\sigma}w_{\tau}$ has a reduced expression starting with $s$
(hence $s\in\desc(w_{\sigma}w_{\tau})$) or (2)
$s\in\sigma\cap\tau$ in which case the factors of $s$ in
$w_{\sigma}$ and $w_{\tau}$ cancel in the product
$w_{\sigma}w_{\tau}$ (hence $s\not\in\desc(w_{\sigma})$).  In either
case, $\sigma\neq\desc(w_{\sigma}w_{\tau})$. Conversely, if
$\St(\bsigma)\cap\btau=\emptyset$ then (1) the product
$\prod_{s\in\sigma}s\prod_{s\in\tau}s$ is a reduced expression
for $w_{\sigma}w_{\tau}$ (there is no cancellation) and (2) no
element of   $\tau$ can be moved to the front of the product.
Thus the set of $s\in S$ that shorten $w_{\sigma}w_{\tau}$ is
precisely $\sigma$.  

To finish the proof note that $w_{\sigma}C_{\tau}\subset C_{\sigma}$
is equivalent to $w_{\sigma}w_{\tau}C\subseteq C_{\sigma}$, which by
Proposition~\ref{prop:descent-test} is equivalent to
$\desc(w_{\sigma}w_{\tau})=\sigma$.  
\end{proof}

As before, we let $\chi$ denote the characteristic series for the greedy
normal form for $(W,S)$.  Again, we fix a total ordering on $N$ (with
$\emptyset$ as the first element in the order). By
Proposition~\ref{prop:star} and Theorem~\ref{thm:cox-rat}, we have  
\[\chi=A(I-Q)^{-1}B\]
where
 \[A=\left[\begin{array}{cccc}1 & 1 & \cdots &
  1\end{array}\right]\hspace{.5in}
\mbox{and}\hspace{.5in}
B=\left[\begin{array}{cccc}1 & 0 & \cdots &
  0\end{array}\right]^T\]
and $Q$ is the $N\times N$-matrix 
\[Q_{\sigma,\tau}=\left\{\begin{array}{ll}
\sigma & \mbox{if $\sigma\neq\emptyset$ and
  $\St(\bsigma)\cap\btau=\emptyset$}\\ 
0 & \mbox{otherwise}\end{array}\right.\]

\begin{example}\label{ex:pent}
Let $W$ be the group generated by reflections across the sides of a
right-angled pentagon.  That is, $W$ has $5$ generators
$s_1$,$s_2$,$s_3$,$s_4$,$s_5$, and the only relations specify that
each $s_i$ is an involution and that $s_i$ and $s_{j}$ commute
whenever $j=i+1 \bmod 5$.  Figure~\ref{fig:pent} shows the Tits cone
(on the left) and the nerve $N$ (on the right).   
\begin{figure}[ht]
\begin{center}
\psfrag{0}{$\emptyset$}
\psfrag{1}{$1$}
\psfrag{2}{$2$}
\psfrag{3}{$3$}
\psfrag{4}{$4$}
\psfrag{5}{$5$}
\psfrag{12}{$12$}
\psfrag{23}{$23$}
\psfrag{34}{$34$}
\psfrag{45}{$45$}
\psfrag{15}{$15$}
\psfrag{N}{$N$}
\includegraphics[scale = .8]{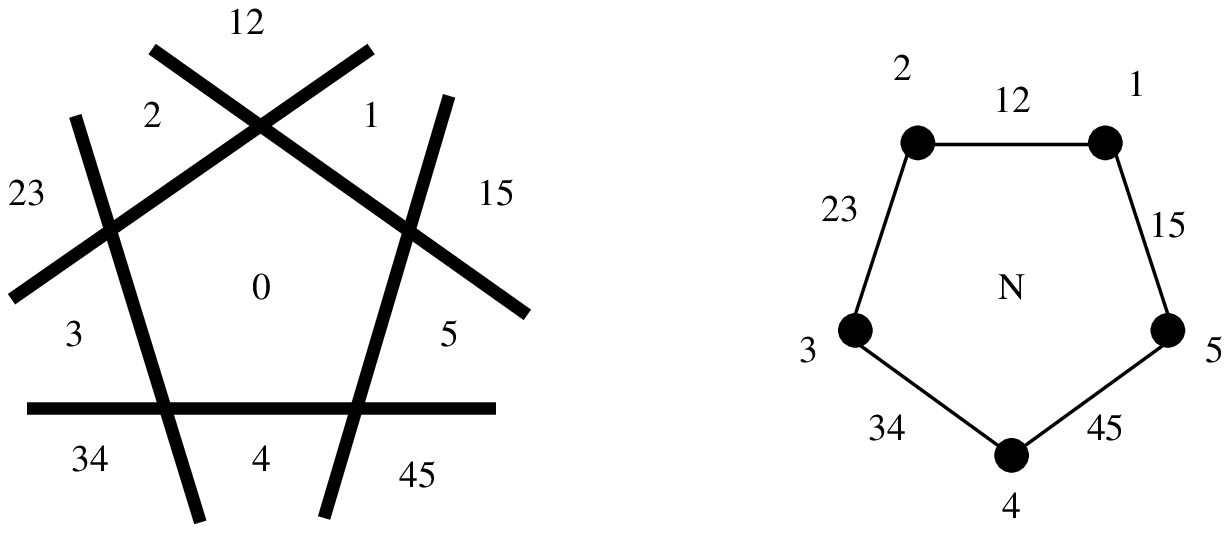}
\caption{\label{fig:pent}}
\end{center}
\end{figure}
The nerve $N$ consists of the subsets $\emptyset$, $\sigma_1=\{s_1\}$,
$\sigma_2=\{s_2\}$, $\sigma_3=\{s_3\}$, $\sigma_4=\{s_4\}$,
$\sigma_5=\{s_5\}$, $\sigma_{12}=\{s_1,s_2\}$, $\sigma_{15}=\{s_1,s_5\}$,
$\sigma_{23}=\{s_2,s_3\}$, $\sigma_{34}=\{s_3,s_4\}$,
and $\sigma_{45}=\{s_4,s_5\}$.
With respect to this order, the matrix $Q$ in this case is 
\[Q=\left[\begin{array}{ccccccccccc}
0&0&0&0&0&0&0&0&0&0&0\\
\sigma_1&0&0&\sigma_1&\sigma_1&0&0&0&0&\sigma_1&0\\
\sigma_2&0&0&0&\sigma_2&\sigma_2&0&0&0&0&\sigma_2\\
\sigma_3&\sigma_3&0&0&0&\sigma_3&0&\sigma_3&0&0&0\\
\sigma_4&\sigma_4&\sigma_4&0&0&0&\sigma_4&0&0&0&0\\
\sigma_5&0&\sigma_5&\sigma_5&0&0&0&0&\sigma_5&0&0\\
\sigma_{12}&0&0&\sigma_{12}&\sigma_{12}&\sigma_{12}&0&0&0&\sigma_{12}&\sigma_{12}\\
\sigma_{15}&0&\sigma_{15}&\sigma_{15}&\sigma_{15}&0&0&0&\sigma_{15}&\sigma_{15}&0\\
\sigma_{23}&\sigma_{23}&0&0&\sigma_{23}&\sigma_{23}&0&\sigma_{23}&0&0&\sigma_{23}\\
\sigma_{34}&\sigma_{34}&\sigma_{34}&0&0&\sigma_{34}&\sigma_{34}&\sigma_{34}&0&0&0\\
\sigma_{45}&\sigma_{45}&\sigma_{45}&\sigma_{45}&0&0&\sigma_{45}&0&\sigma_{45}&0&0\end{array}\right].\]
In principle, one could obtain a rational expression for $\chi$ by
formally inverting the matrix $I-Q$ using row-reduction.  The
resulting expression would be rather unwieldy.  On the other hand, if
we let $\bbQ[[\bt]]$ denote the multivariable power series ring with
{\em commuting} parameters $\bt=(t_{\sigma})$ (in this case, $10$
parameters), then we can consider the image $\chi^{abel}$ of $\chi$
under the substitution homomorphism
$\bbQ\<\<\cA\>\>\rightarrow\bbQ[[\bt]]$ that takes each $\sigma$ to
the corresponding $t_{\sigma}$.  On general principles, the resulting
power series will be a rational function of the form 
\[\chi^{abel}(\bt)=\frac{p(\bt)}{q(\bt)}\]
which can be computed easily using an computer algebra package.  In
this case, the polynomials $p$ and $q$ we obtained both have degree
$10$, $p$ with $244$ terms and $q$ with $154$.  A more
manageable substitution is to put every $\sigma_i=x$ and every
$\sigma_{ij}=y$, and treat $x$ and $y$ as commuting variables.  In this
case, the resulting series $\bar{\chi}(x,y)$ has rational expression
\[\bar{\chi}(x,y)=\frac{xy+3x+3y+1}{xy-2x-2y+1}.\]   
\end{example}

Our first result gives a condition under which the reciprocal
$\chi^*$ exists and can be explicitly computed.  As in the previous
section, we let
$J\in\bbQ^{n\times n}$ be the matrix given by  
\[J_{\sigma,\tau}=\left\{\begin{array}{ll}
(-1)^{\dim\bsigma} & \mbox{if $\St(\bsigma)\cap\btau=\emptyset$}\\
0 & \mbox{otherwise}\end{array}\right..\]
Then the matrix $Q$ has a factorization of the form 
\begin{eqnarray}
Q=D_0J
\end{eqnarray}
where $D_0$ is a quasi-regular diagonal matrix. More precisely, let
$D\in\bbQ\<K\>^{n\times n}$ be the diagonal matrix given by  
\[D_{\sigma,\sigma}=\left\{\begin{array}{ll}
1 & \mbox{if $\sigma=\emptyset$, and }\\
(-1)^{\dim\bsigma}\sigma & \mbox{otherwise}
\end{array}\right..\]
Then $D_0$ is the matrix obtained from $D$ by replacing the first
entry with $0$.  We then have the following.

\begin{lemma}\label{lem:formula}
Let $(W,S)$ be a right-angled Coxeter group, let $N$ denote the nerve,
and let $\chi$ denote the characteristic series of the greedy normal
form.  If the matrix $J=J(N)$ is invertible over $\bbQ$, then the
reciprocal $\chi^*$ exists and is given by the formula 
\[\chi^*=-\frac{1}{\bar{\bchi}(N)}A(I-D_0J^{-1})^{-1}B.\]
In particular, $\chi^*$ is a power series in $\bbQ\<\<\cA\>\>$.
\end{lemma}
 
\begin{proof}{}
First we note that $I-\bar{Q}$ can be written as
\[I-\bar{Q}=I-\bar{D_0}J=\bar{D}D-\bar{D_0}J.\]
Letting $P$ denote the square matrix with all zeros except a $1$ in
the top left entry, we then have
\[\bar{D}D-\bar{D_0}J=(\bar{D}D-P)-(\bar{D_0}J-P)=\bar{D}D_0-\bar{D}J.\]
Combining these equations, we obtain  
\[I-\bar{Q}=\bar{D}D_0-\bar{D}J=\bar{D}(-J+D_0).\]
If $J$ is invertible, this gives  
\[I-\bar{Q}=(-\bar{D})(I-D_0J^{-1})J.\]
Since $D^{-1}=\bar{D}$ and $D_0J^{-1}$ has quasi-regular entries, $I-\bar{Q}$ is invertible with inverse given by 
\[(I-\bar{Q})^{-1}=J^{-1}(I-D_0J^{-1})^{-1}(-D).\]
Substituting this into 
\[\chi^*=A(I-\bar{Q})^{-1}B\]
gives 
\[\chi^*=AJ^{-1}(I-D_0J^{-1})^{-1}(-DB).\]
By Theorem~\ref{thm:j}, we have $AJ=\bar{\bchi}(N)A$. Multiplying by
$J^{-1}$ on the right and dividing by $\bar{\bchi}(N)$ (which is
necessarily nonzero since $J$ is invertible) gives
$AJ^{-1}=\frac{1}{\bar{\bchi}(N)}A$.  Substituting this and $-DB=-B$ into
the previous expression for $\chi^*$ gives the desired formula.
\end{proof}

\begin{example}
Let $(W,S)$ be the infinite dihedral group $D_{\infty}$.  That is, $S$
has two noncommuting generators $\{x,y\}$, and the nerve is the
simplicial complex consisting of two distinct points (i.e., the nerve
is the $0$-dimensional sphere).  The matrix $Q$ and the calculation of
$\chi$ and $\chi^*$ were already illustrated in Examples~\ref{ex:square-free}
and \ref{ex:square-free-2}. Note that in this case $\chi^*=-\chi$.
\end{example}

\begin{example}
Let $(W,S)$ be the free product $\bbZ_2\ast\bbZ_2\ast\bbZ_2$.  This
time $S$ consists of $3$-elements $\{x,y,z\}$, no two of which
commute.  The nerve is $3$ points and the matrices $Q$, $J$, and $D$,
are given by:
\[Q=\left[\begin{array}{cccc}
0&0&0&0\\
x&0&x&x\\
y&y&0&y\\
z&z&z&0\end{array}\right],\;
J=\left[\begin{array}{cccc}
-1&0&0&0\\
1&0&1&1\\
1&1&0&1\\
1&1&1&0\end{array}\right],\;
D=\left[\begin{array}{cccc}
1&0&0&0\\
0&x&0&0\\
0&0&y&0\\
0&0&0&z\end{array}\right].\]
Since $J$ is invertible with inverse
\[J^{-1}=\frac{1}{2}\left[\begin{array}{cccc}
-2&0&0&0\\
1&-1&1&1\\
1&1&-1&1\\
1&1&1&-1\end{array}\right],\]
the reciprocal $\chi^*$ exists (by Lemma~\ref{lem:formula}) and is given 
by  
\begin{eqnarray*}
\chi^*& = & -\frac{1}{2}A(I-D_0J^{-1})^{-1}B\\
&=& -\frac{1}{2}\left[\begin{array}{cccc}1&1&1&1\end{array}\right] 
\left(I-\frac{1}{2}\left[\begin{array}{cccc}0&0&0&0\\x&-x&x&x\\y&y&-y&y\\z&z&z&-z\end{array}\right]\right)^{-1}\left[\begin{array}{c}1\\0\\0\\0\end{array}\right].\\
&=&-\frac{1}{2}-\frac{1}{4}(x+y+z)-\frac{1}{8}(xy+xz+yx+yz+zx+zy-x^2-y^2-z^2)-\cdots
\end{eqnarray*}
\end{example} 

We can now prove the main theorem.

\begin{theorem}{\bf (Reciprocity)}\label{thm:cox-rec} If $(W,S)$ is a right-angled
  Coxeter group and the nerve $N$ is an Eulerian $k$-sphere, then
  $\chi^*$ exists and  
\[\chi^*=(-1)^{k+1}\chi.\] 
\end{theorem}

\begin{proof}{}
By Lemma~\ref{lem:formula}, $\chi^*$ exists and is given by 
\[\chi^*=-\frac{1}{\bar{\bchi}(N)}A(I-D_0J^{-1})^{-1}B.\]
Since $N$ is an Eulerian $k$-sphere, $\bar{\bchi}(N)=(-1)^{k}$ and 
$J$ is an involution (by Theorem~\ref{thm:j}). Thus, we have 
\[\chi^*=-(-1)^{k}A(I-D_0J)^{-1}B=(-1)^{k+1}A(I-Q)^{-1}B=(-1)^{k+1}\chi.\]
\end{proof}

\begin{example}
The right-angled pentagonal Coxeter group considered in
Example~\ref{ex:pent} has nerve equal to the boundary complex of a
pentagon, hence is an Eulerian $1$-sphere.  By
Theorem~\ref{thm:cox-rec}, the characteristic series $\chi$ satisfies
$\chi^*=\chi$.  It follows that the rational function
$\chi^{abel}(\bt)=\frac{p(\bt)}{q(\bt)}$ described in
Example~\ref{ex:pent} satisfies the reciprocity formula   
\[\chi^{abel}(\bt^{-1})=\chi^{abel}(\bt)\]
where the expression on the left-side denotes the rational function
obtained by replacing each commuting parameter $t_{\sigma}$ in
$\chi^{abel}(\bt)$ with its reciprocal $1/t_{\sigma}$.  This formula
can be verified easily with a computer algebra package.  For an even
simpler illustration of reciprocity, we can use the series
$\bar{\chi}(x,y)$ obtained by the substitution $\sigma_i=x$,
$\sigma_{ij}=y$.  In this case, the resulting series has rational
expression  
\[\bar{\chi}(x,y)=\frac{xy+3x+3y+1}{xy-2x-2y+1},\]   
which clearly satisfies the reciprocity formula
\[\bar{\chi}(x^{-1},y^{-1})=\bar{\chi}(x,y).\]
\end{example}

The following example shows that the right-angled assumption in
Theorem~\ref{thm:cox-rec} cannot be removed.
 
\begin{example}
Consider the affine Coxeter group $\tilde{B}_2$.  The nerve $N$ is the
boundary complex of a triangle.  Though not a flag complex, $N$ {\em
  is} an Eulerian sphere, and the usual growth series with respect to
the standard generators has rational expression
\[\gamma(t)=\frac{t^4+2t^3+2t^2+2t+1}{t^4-t^3-t+1}\]
Since $\tilde{B_2}$ is an affine Coxeter group, this example falls
into Serre's original class of examples for which the reciprocity
formula $\gamma(1/t)=\gamma(t)$ holds.  On the other hand, the
characteristic series $\chi$ for the greedy normal form {\em does not}
satisfy the reciprocity formula $\chi^*=\chi$.  If it did, then
substituting the single variable $t$ for each $\sigma\in\cA$ in $\chi$
would result in a series $\bar{\chi}(t)$ which satisfies the
reciprocity formula. But as we computed in Example~\ref{ex:b2}, this
series has rational representation  
\[\bar{\chi}=\frac{1+4t-3t^2+4t^3-2t^4}{1-2t+t^2},\]
and this does not satisfy $\bar{\chi}(1/t)=\bar{\chi}(t)$. 
\end{example}

\section{Applications to growth series}\label{s:growthseries}

{\bf Standard generators.}
Let $W$ be a Coxeter group with generating set $S$, and let
$i:S\rightarrow I$ be any function that is constant on conjugacy
classes.  Let $\bt$ denote the $I$-tuple $(t_i)_{i \in I}$, and let
$\bbQ[[\bt]]$ denote the ring of formal power series in the commuting
variables $t_i$.  For any $w\in W$, let $s_1\cdots s_n$ be a reduced
expression for $w$.  Then the monomial $\prod_{j=1}^n
t_{i(s_j)}\in\bbQ[[\bt]]$ 
does not depend on the choice of reduced expression (this follows from
Tits' solution to the word problem -- see, e.g., \cite{Davis-book}).
We denote this monomial by $t^w$.  We then define the {\em standard
  growth series} $\gamma(\bt)\in \bbQ[[\bt]]$ by 
\[\gamma(\bt)=\sum_{w\in W} t^w.\]

Similarly, we let $\bbQ[W][[\bt]]$ denote the ring of formal power
series with commuting parameters $\bt$ and coefficients in the
(noncommutative) group ring $\bbQ[W]$.  The {\em standard complete
  growth series} $\tilde{\gamma}(\bt)\in \bbQ[W][[\bt]]$ is then defined
by 
\[\tilde{\gamma}(\bt)=\sum_{w\in W}w t^w.\]

Note that when $I$ is a singleton set, $\gamma$ (respectively,
$\tilde{\gamma}$) is the ordinary single-variable (resp., complete)
growth series.  At the other extreme, if $W$ is right-angled (or more
generally if all $m(s,s')$ are even), then no 
two generators in $S$ are conjugate in $W$, so we can take $I=S$ and
$i:S\rightarrow S$ to be the identity map.  In this case, we have a
parameter $t_s$ for each generator $s\in S$.

Now let $\cA$ be the proper nerve of $W$, and let $\bbQ\<\<\cA\>\>$
denote the corresponding power series ring.  
We define homomorphisms 
$\phi:\bbQ\<\<\cA\>\>\rightarrow\bbQ[[\bt]]$ and 
$\tilde{\phi}:\bbQ\<\<\cA\>\>\rightarrow\bbQ[W][[\bt]]$ by 
$\phi(\sigma)=t^{w_{\sigma}}$ and
$\tilde{\phi}(\sigma)=w_{\sigma}t^{w_{\sigma}}$.  Let
$\cL\subseteq\cA^*$ be the greedy normal form for $W$ and let $\chi$
be its characteristic series.  Then 
\[\phi(\chi)=\gamma(\bt)\] 
and   
\[\tilde{\phi}(\chi)=\tilde{\gamma}(\bt).\]
Since $\phi$ and $\tilde{\phi}$ take rational series to rational
series, Theorem~\ref{thm:cox-rat} gives the following.

\begin{corollary}
For any Coxeter group, both of the growth series $\gamma(\bt)$ and
$\tilde{\gamma}(\bt)$ are rational.
\end{corollary}

Since $\phi$ and $\tilde{\phi}$ are homomorphisms,
Theorem~\ref{thm:cox-rec} gives the following. 

\begin{corollary}
Let $W$ be a right-angled Coxeter group and assume the nerve is an
Eulerian $k$-sphere.  Then 
\[\gamma(\bt^{-1})=(-1)^{k+1}\gamma(\bt)\]
and 
\[\tilde{\gamma}(\bt^{-1})=(-1)^{k+1}\tilde{\gamma}(\bt)\]
where $\bt^{-1}$ denotes the $I$-tuple $(t_i^{-1})_{i\in I}$.
\end{corollary}

Most of these rationality and reciprocity formulas are well-known,
with the possible exception of the reciprocity formula for the
complete growth series.

{\bf Greedy generators.}
Let $W$ be a Coxeter group and let $A\subset W$ be the generating set
$A=\{w_{\sigma}\;|\;\sigma\in\cA\}$.  We define the (single variable)
{\em (left) greedy growth series} $\gamma_A(t)\in\bbQ[[t]]$ by  
\[\gamma_A(t)=\sum_{w\in W}t^{|w|}\]
where $|w|$ now denotes the word length of $w$ {\em with respect to $A$}.
This series $\gamma_A(t)$ was considered by the author and R.~Glover
in \cite{GS} under the name ``automatic growth series''.  (In fact, it
was the observation that rational expressions for some of these series 
had palindromic numerator and denominator which suggested a more general
reciprocity formula.)  Similarly, one defines the {\em greedy complete
  growth series} by  
\[\tilde{\gamma}_A(t)=\sum_{w\in W}wt^{|w|}.\]

For general Coxeter groups, the length (in $\cA^*$) of a greedy word
$\alpha$ need not coincide with the length of $w=\pi(\alpha)$.   

\begin{example} Let $W$ be the dihedral group of order $8$ with generators
$s,t$.  Then $\cA=\{\{s\},\{t\},\{s,t\}\}$ and $A=\{s,t,stst\}$.  The
element $w=sts$ has $\alpha=\{s\}\{t\}\{s\}$ as its greedy
representative, but can be written as a product of just two
generators: $w=(stst)(t)$.  
\end{example}

However, if $W$ is right-angled and $\alpha\in\cL$, then the length of
$\alpha$ in $\cA^*$ does coincide with $|\pi(\alpha)|$.  This means
that for right-angled Coxeter groups, summing $t^{|w|}$ over all $w\in
W$ is the same as summing $t^{l(\alpha)}$ over all
$\alpha\in\cL$ (where $l(\alpha)$ denotes the length of $\alpha$ as a
word in $\cA^*$).   Hence, if we define homomorphisms
$\phi_A:\bbQ\<\<\cA\>\>\rightarrow\bbQ[[t]]$ and 
$\tilde{\phi}_A:\bbQ\<\<\cA\>\>\rightarrow\bbQ[W][[t]]$ 
by $\phi_A(\sigma)=t$ and by $\tilde{\phi}_A(\sigma)=w_{\sigma}t$,
respectively, then we get 
\[\phi_A(\chi)=\gamma_A(t)\]
and 
\[\tilde{\phi}_A(\chi)=\tilde{\gamma}_A(t).\]
Now applying our main theorems gives the following.

\begin{corollary}
Let $W$ be a right-angled Coxeter group, and let $\gamma_A(t)$ and
$\tilde{\gamma}_A(t)$ denote, respectively, the greedy growth series
and complete greedy growth series of $W$.  Then $\gamma_A(t)$ and
$\tilde{\gamma}_A(t)$ are both rational series.  Moreover, if the
nerve of $W$ is an Eulerian $k$-sphere then
\[\gamma_A(t^{-1})=(-1)^{k+1}\gamma_A(t)\]
and 
\[\tilde{\gamma}_A(t^{-1})=(-1)^{k+1}\tilde{\gamma}_A(t).\]
\end{corollary}

\bibliographystyle{plain}
\bibliography{biblio}

\end{document}